\documentclass[11pt]{amsart}

\usepackage{amsmath}
\usepackage{amssymb}
\usepackage{mathrsfs}
\usepackage{xcolor, enumitem, comment, todonotes}
\usepackage[all]{xy}
 \usepackage{url}
 \usepackage{amsthm}
 \usepackage{thmtools}
\usepackage{thm-restate}
\usepackage{hyperref,soul}

\theoremstyle{plain}
\newtheorem*{theorem*}{Theorem}

\newtheorem{theorem}{Theorem}[section]
\newtheorem{claim}[theorem]{Claim}

\newtheorem{lemma}[theorem]{Lemma}

\newtheorem{proposition}[theorem]{Proposition}
\newtheorem{corollary}[theorem]{Corollary}

\newtheorem{observation}[theorem]{Observation}

\theoremstyle{definition}
\newtheorem{definition}[theorem]{Definition}

\newtheorem{problem}{Problem}[section]

\newtheorem{question}[theorem]{Question}

\theoremstyle{remark}
\newtheorem{remark}[theorem]{Remark}

\newtheorem{notation}[theorem]{Notation}
\newtheorem{fact}[theorem]{Fact}

\def\l{{\langle}}
\def\r{{\rangle}}
\newcommand{\MS}{\mathcal{MS}}  

\newcount\skewfactor
\def\mathunderaccent#1#2 {\let\theaccent#1\skewfactor#2
\mathpalette\putaccentunder}
\def\putaccentunder#1#2{\oalign{$#1#2$\crcr\hidewidth
\vbox to.2ex{\hbox{$#1\skew\skewfactor\theaccent{}$}\vss}\hidewidth}}

\newcommand{\lusim}[1]{\smash{\underset{\raisebox{1.2pt}[0cm][0cm]{$\sim$}}
{{#1}}}}
\def\smallbox#1{\leavevmode\thinspace\hbox{\vrule\vtop{\vbox
   {\hrule\kern1pt\hbox{\vphantom{\tt/}\thinspace{\tt#1}\thinspace}}
   \kern1pt\hrule}\vrule}\thinspace}

\newcommand\cat[1]{{}^\curvearrowright\langle #1\rangle}
 
\DeclareMathOperator{\cof}{cof}

\newcommand{\cf}{{\rm cf}}
\def\Ult{{\rm Ult}}

\title[Transferring compactness]{Transferring compactness}
\date{\today}

\author{Tom Benhamou}
\thanks{The research of the first author was supported by the National Science Foundation under Grant
No. DMS-2346680}
\address[Benhamou]{Department of Mathematics, Statistics, and Computer Science, University of Illinois at Chicago, Chicago, IL 60607, USA}
\email{tomb@uic.edu}

\author{Jing Zhang}
\thanks{The research of the second author was supported by NSERC discovery grants}
\address[Zhang]{Department of Mathematics, University of Toronto, 40 St. George St., Toronto, ON, Canada}
\email{jingzhan@alumni.cmu.edu}


\begin{document}

\begin{abstract}
    We demonstrate that the technology of Radin forcing can be used to transfer compactness properties at a weakly inaccessible but not strong limit cardinal to a strongly inaccessible cardinal.
    As an application, relative to the existence of large cardinals, we construct a model of set theory in which there is a strongly inaccessible cardinal $\kappa$ that is $n$-$d$-stationary for all $n\in \omega$ but not weakly compact. This is in sharp contrast to the situation in the constructible universe $L$, where $\kappa$ being $(n+1)$-$d$-stationary is equivalent to $\kappa$ being $\mathbf{\Pi}^1_n$-indescribable. We also show that it is consistent that there is a cardinal $\kappa\leq 2^\omega$ such that $P_\kappa(\lambda)$ is $n$-stationary for all $\lambda\geq \kappa$ and $n\in \omega$, answering a question of Sakai.
    \end{abstract}

\maketitle
\section{Introduction}
In general, \textit{compactness} refers to the phenomenon that if some property holds for all small substructures then it holds for the structure itself. For example, a compact topological space asserts that any collection of closed sets with the finite intersection property, has a non-empty intersection; The compactness theorem for first order Logic states that any first order theory such that all of its finite subsets are consistent must also be consistent; In cardinal arithmetic, Silver's theorem \cite{Silver} asserts that if $2^{\aleph_\alpha}=\aleph_{\alpha+1}$ for any $\alpha<\omega_1$, then necessarily $2^{\aleph_{\omega_1}}=\aleph_{\omega_1+1}$. This compactness phenomenon does not occur  at the level of $\aleph_\omega$, as Magidor proves \cite{MagAnnals} that it is consistent that $2^{\aleph_n}=\aleph_{n+1}$ for every $n<\omega$ while $2^{\aleph_\omega}>\aleph_{\omega+1}$. In Graph Theory, K\"{o}nig's Lemma asserts that if $G$ has an infinite, locally finite, and connected graph, then there is an infinite simple path. This lemma ensures for example that $\omega$ has the \textit{tree property} which is a paradigmatic compactness principle which says that any countably infinite tree, such that every level is finite must have a branch.

The dual notion of compactness is \textit{reflection} i.e. if some property holds at some mathematical structure, then there must be a small substructure for which it was true. So compactness of some property $\phi$ is equivalent to the reflection of $\neg\phi$. 
It turns out that many important instances of compactness such as free of abelian groups, Metrizable topological spaces \cite{MagidorShelah1994} and others \cite{Shelah2012OnIF,Todorcevic1983OnAC} boil down to a specific reflection principle known as \textit{stationary reflection}. Recall that a subset $C\subseteq \kappa$ is a \textit{closed unbounded} (club) if $C$ is closed in the order topology of $\kappa$ and in unbounded below $\kappa$. A subset $S\subseteq \kappa$ is \textit{stationary} if it intersects any club.
\begin{definition}
    We say that a cardinal $\kappa$ satisfies sationary reflection if for any stationary set $S\subseteq \kappa$, there is $\alpha<\kappa$ of uncountable cofinality such that $S\cap\alpha$ is stationary at $\alpha$.
\end{definition}
Usually, reflection principles require assumptions beyond $ZFC$ i.e. \textit{large cardinals}. In fact, some large cardinal notions are tailored to satisfy reflection and compactness properties e.g. \textit{weakly/strongly/super-compact cardinals}. One specific hierarchy of large cardinals which this paper considers is the $\Pi^1_n$-indescribable cardinals (see definition \ref{Def: indescribable}) which was discovered by Hanf and Scott \cite{HanfScott}. These large cardinals turned out to form a yardstick hierarchy in the landscape of large cardinals and provide a nice characterization of other large cardinal notions in terms of their ability to reflect formulas of higher complexity. 
Due to lack of technologies, a few implications among certain compactness principles around the region of ``moderate large cardinals'' are not well understood.

For example, it is open whether $\kappa>\omega$ being weakly compact is implied by any of the following: 
\begin{enumerate}
    \item any two $\kappa$-c.c posets $P, Q$, $P\times Q$ is also $\kappa$-c.c,
    \item $\kappa\to [\kappa]^2_\omega$,
    \item $\kappa$ is strongly inaccessible and there does not exist a $\kappa$-Suslin tree,
     \item $\kappa$ is strongly inaccessible and $\kappa\to [\kappa]^2_\kappa$,
     \item $\kappa$ is strongly inaccessible J\'{o}nsson, namely, $\kappa\to [\kappa]^{<\omega}_\kappa$.
\end{enumerate}
The first 4 items are consequences of $\kappa$ being weakly compact while the last item is not.

It is important that we insist $\kappa$ is a strongly inaccessible cardinal in the last 3 items since these properties are consistent with $\kappa$ being weakly inaccessible but not strong limit. Since if $\kappa$ is weakly compact, then it is necessarily a strong limit cardinal, we can cheat and declare these principles are separated. However, if we insist that $\kappa$ is strongly inaccessible, then these problems become much harder. In fact, they are open.

In this paper, we explore the possibility of ``fixing the cheat'' by transferring compactness principles at a weakly inaccessible cardinal to a strongly inaccessible cardinal. The technology we employ is Radin forcing \cite{Radin}, denoted $R_{\bar{U}}$, which is defined using a measure sequence $\bar{U}$ on a cardinal $\kappa$. Radin forcing has already turned out useful in order to tune the large cardinal properties and compactness principles holding at $\kappa$ in the model $V^{R_{\bar{U}}}$. For example, \begin{enumerate}
    \item If $cf(lh(\bar{U}))=\rho<\kappa$ then $V^{R_{\bar{U}}}\models cf(\kappa)=cf(\rho)$ (\cite[Section 5.1]{Gitik2010}).
    \item If $cf(lh(\bar{U}))\geq\kappa^+$ then $V^{R_{\bar{U}}}\models \kappa$ is strongly inaccessible (\cite{MITCHELLHowWeak}).
    \item If $\bar{U}$ has a \textit{repeat point} then $V^{R_{\bar{U}}}\models \kappa$ is measurable (\cite{MITCHELLHowWeak}).
    \item If $\kappa^+\leq cf(lh(\bar{U}))\leq lh(\bar{U})<2^\kappa$, then $V^{R_{\bar{U}}}\models \neg \diamondsuit_{\kappa}$ (Woodin, see \cite{Omer1}).
    \item If $\bar{U}$ has a \textit{weak repeat point} then $V^{R_{\bar{U}}}\models \kappa$ is weakly compact (\cite{Omer1}).
    \item If $cf(lh(\bar{U}))\geq\kappa^{++}$ then $V^{R_{\bar{U}}}\models$ $\kappa$ satisfy stationary reflection (\cite{Omer1}).
    \item If $\bar{U}$ satisfy the \textit{local repeat point} then $V^{R_{\bar{U}}}\models$ $\kappa$ is almost inaffable (\cite{OmerJing}).
\end{enumerate}    
The common idea in those results is that we isolate some property of the length  $lh(\bar{U})$ of the measure sequence $\bar{U}$ which guarantees that $\kappa$ has some large cardinal property in $V^{R_{\bar{U}}}$. Let us just mention that most of the implications above are reversible. In \cite{OmerJing}, Ben-Neria and the second author tighten the connection between compactness principles in the Radin extension and properties.
In this paper, the length of the sequence is usually $lh(\bar{U})<(2^\kappa)^+$. The very rough idea is that if we force using a measure sequence such that the length of the measure sequence satisfies suitable compactness principles, then the Radin forcing transfers these compactness principles to actually hold at $\kappa$, which is strongly inaccessible in the generic extension. 

As an application, relative to the existence of large cardinals, we construct a model where higher order stationary reflections hold at a strongly inaccessible cardinal which is not weakly compact. To properly state the theorem, we need the following definitions. Bagaria \cite[Definition 4.1]{Bagaria} used generalized logic to extend the indescribable cardinal hierarchy of Hanf and Scott to $\mathbf{\Pi}^1_{\xi}$-formulas for $\xi\geq \omega$.

\begin{definition}[Hanf-Scott for $\xi\in \omega$, Bagaria \cite{Bagaria} for $\xi\geq \omega$]\label{Def: indescribable}
Let $\xi$ be an ordinal.
A set $S\subseteq \kappa$ is $\mathbf{\Pi}^1_{\xi}$-indescribable if for
all $R \subseteq V_\kappa$ and all $\mathbf{\Pi}^1_\xi$-sentence $\phi(X)$, if $(V_\kappa , \in, R)\models \phi(R)$ then there is an $\alpha\in S$ such that
$(V_\alpha, \in, R \cap V_\alpha) \models \phi(R\cap V_\alpha)$.
\end{definition}

\begin{definition}[Bagaria \cite{Bagaria}]\label{definition: higherorderstat}
Recursively define that a set $A$ is:
\begin{enumerate}
    \item $0\text{-stationary in }\alpha\text{ if }\sup(A)=\alpha$,
    \item $\xi\text{-stationary in }\alpha\text{ where $\xi\leq\alpha$ if }$ $$\forall \eta<\xi\forall S\text{ which is }\eta\text{-stationary in }\alpha,\exists\beta\in A, S\cap\beta\text{ is  }\eta\text{-stationary}$$
    \item Given $A\subset \kappa$, let $Tr_{\xi}(A)$ denote the set $$\{\alpha\in \kappa: A\cap \alpha \text{ is }\xi\text{-stationary}\}.$$
\end{enumerate}

We say that $\alpha$ is $\xi$-stationary if $\alpha$ is $\xi$-stationary as a subset of $\alpha$.  
\end{definition}
Bagaria's motivation for the notions comes from a result in \cite{Bagaria}, where these higher-order stationary reflection properties characterize the non-isolated points in the ordinal topology interpretation of
generalized provability logics (see \cite{GLP} or \cite{provability} for more information regarding this motivation).
Note that:
\begin{enumerate}
    \item $A$ is $1$-stationary iff $A$ is stationary,
    \item $\alpha$ is $1$-stationary iff $\alpha$ has uncountable cofinality,
    \item $\alpha$ is $2$-stationary iff every stationary subset of $\alpha$ reflects.
\end{enumerate}

Let us define the two variations of Bagaria's higher order stationarity central to this paper. Loosely speaking, one is obtained by varying the degree of simultaneous reflection and the other one is the diagonal version. 

\begin{definition}\label{definition: (n,chi)-s-stationary}
    Let $2\leq \chi<\kappa$ be any regular cardinal and $\kappa$ be a limit ordinal.
    \begin{enumerate}
        \item  $A\subseteq \kappa$ is called $(0,\chi)$-$s$-stationary iff $A$ is unbounded in $\kappa$ and $\cf(A)=\cf(\kappa)\geq \chi$.
        \item   $A\subseteq \kappa$ is called $(\xi,\chi)$-$s$-stationary if for any $\chi'<\chi$ and any $\l T_i\mid i<\chi'\r$ such that each $T_i$ is $(\eta_i, \chi)$-stationary for some $\eta_i<\xi$, there is $\alpha\in A$ such that $\forall i<\chi'$, $T_i\cap\alpha$ is $(\eta_i,\chi)$-stationary.
        \item Given $A\subset \kappa$, let $Tr^\chi_{\xi}(A)$ denote the set $$\{\alpha\in \kappa: A\cap \alpha \text{ is }(\xi,\chi)\text{-$s$-stationary}\}.$$
      \end{enumerate}

\end{definition}
\begin{remark} 
    \begin{itemize}
        \item $\kappa$ is $(1,\chi)$-s-stationary iff $\cf(\kappa)>\chi$.
        \item $S\subseteq \kappa$ is $(1,\chi)$-$s$-stationary iff $S \cap \cof(\geq\chi)\cap \kappa$ is a stationary subset of $\kappa$.
        \item  $\kappa$ is $(2,\chi)$-s-stationary iff every less than $\chi$-many stationary subsets of $\cof(\geq\chi)\cap \kappa$ reflect simultaneously. 
    \end{itemize}
\end{remark}

\begin{remark}
The case when $\chi=3$ appeared in \cite[Definition 2.8]{Bagaria}. We will follow the existing literature and let ``$n$-$s$-stationary'' denote ``$(n,3)$-$s$-stationary'' as in Defintion \ref{definition: (n,chi)-s-stationary}.
\end{remark}

\begin{definition}\label{definition: n-d-stationary}
    Let $\kappa$ be a ordinal.
    \begin{enumerate}
        \item  $A\subseteq \kappa$ is called $0$-$d$-stationary iff $A$ is unbounded in $\kappa$.
        \item  $A$ is called $\xi$-$d$-stationary if for every $\l T_i\mid i<\kappa\r$ such that each $T_i$ is $\eta_i$-stationary for some $\eta_i<\xi$, there is some $\alpha\in A$ such that $\forall i<\alpha$, $T_i\cap\alpha$ is $\eta_i$-$d$-stationary.
        \item Given $A\subset \kappa$, let $Tr^d_{\xi}(A)$ denote the set $$\{\alpha\in \kappa: A\cap \alpha \text{ is }\xi\text{-$d$-stationary}\}.$$
       
    \end{enumerate}

\begin{remark} 
    \begin{itemize}
        \item $\kappa$ is $1$-$d$-stationary iff $\kappa$ is a regular cardinal.
        \item $S\subseteq \kappa$ is $1$-$d$-stationary iff $S\cap \kappa$ is stationary.
        \item  $\kappa$ is $2$-$d$-stationary iff for any $\langle T_i: i<\kappa\rangle$ where each $T_i$ is a stationary subset of $\kappa$, there exists a regular $\alpha<\kappa$ such that for all $i<\alpha$, $T_i\cap \alpha$ is stationary in $\alpha$. 
    \end{itemize}
\end{remark}

\begin{theorem}[Jensen \cite{Jensen1972} for $\xi=1$, Bagaria-Magidor-Sakai \cite{BagariaMagidorSakai} for $\xi\in (1,\omega)$, Bagaria \cite{Bagaria} for $\xi\geq \omega$]
  In $L$, the following are equivalent: 
  \begin{enumerate}
      \item  $\alpha$ is $\xi+1$-stationary,
      \item $\alpha$ is $\xi+1$-$s$-stationary,
      \item  $\alpha$ is $\mathbf{\Pi}^1_\xi$-indescribable.
  \end{enumerate}
\end{theorem}

As we will see in Corollary \ref{corollary: equivalence}, the equivalence also extends to $\xi+1$-$d$-stationary and $(\xi+1,\chi)$-$s$-stationary for any $\chi<\alpha$.
Note that the implications from bottom to top are valid in ZFC. The additional constructibility assumption helps in proving (1) implies (4). 

In terms of the consistency strength of these principles, Magidor \cite{Magidor1982} showed that the existence of a $2$-s-stationary is equiconsistent with the existence of a weakly compact cardinal. Surprisingly, Mekler and Shelah \cite{MeklerShelah} showed that the consistency strength of $\kappa$ being 2-stationary is strictly in between a greatly Mahlo cardinal and a weakly compact. They isolated \emph{reflection cardinals} and showed $\kappa$ being a reflection cardinal is equiconsistent with $\kappa$ being a 2-stationary cardinal. Generalizing their results and methods, Bagaria-Magidor-Mancilla \cite{BagariaMagidorMancilla} showed that the consistency strength of a $\xi+1$-stationary cardinal is strictly in between a $\xi$-greatly-Mahlo cardinal and a $\mathbf{\Pi}^1_\xi$-indescribable cardinal. We refer the readers to \cite{BagariaMagidorMancilla} for relevant definitions. To achieve this, they isolate the notion of a $\xi$-reflection cardinal and show that   
\begin{itemize}
    \item there are many $\xi$-reflection cardinals below any $\mathbf{\Pi}^1_\xi$-indescribable cardinal,
    \item no $\xi$-reflection cardinal can be $\leq$ the first $\xi$-greatly Mahlo cardinal.
\end{itemize}

Note that by definition, if $\kappa$ is a $\xi$-reflection cardinal, then $\kappa$ is $\xi$-stationary. In $L$, even more is true: it is $\xi+1$-stationary.

It is therefore a natural question to clarify the relationship between higher order stationary reflections and indescribable cardinals. For example, for any given $\zeta$, is it true that there exists a large enough $\xi$ such that whenever $\kappa$ is $\xi$-$d$-stationary (or $\xi$-stationary), then $\kappa$ is 
$\mathbf{\Pi}^1_\zeta$-indescribable? The main result of this paper is that in general the answer is negative.

Another reason for this investigation is to expose another way of establishing higher order stationary reflection principles, fundamentally different from the Mekler-Shelah approach. Aside from the papers mentioned previously, variations of the Mekler-Shelah method have been used to study the extent of the weakly compact reflection principle by Cody and Sakai \cite{CodySakai}.

The following are the main results for this paper. 
\begin{theorem}\label{theorem: mainnonstronglimit0}
    Suppose that $\lambda$ is a measurable cardinal in $V$. Then in any forcing extension with a poset satisfying $\gamma$-c.c for some $\gamma<\lambda$, $\lambda$ is $\lambda$-stationary, $\lambda$-$d$-stationary and $(\lambda,\chi)$-stationary for all $\chi<\lambda$.
\end{theorem}

In particular, we have a way of producing a non strong limit weakly inaccessible cardinal $\lambda$ that is $\lambda$-$d$-stationary. The next theorem ``transfers'' this compactness to a strongly inaccessible cardinal, using the technology of Radin forcing. 
\begin{theorem}\label{theorem: mainstronglimit}
    Relative to the existence of a $H(\lambda^{++})$-hypermeasurable cardinal\footnote{A cardinal $\kappa$ is an $H(\theta)$-hypermeasurable cardinal if there is an elementary embedding $j:V\rightarrow M$ with $crit(j)=\kappa$ and $H(\theta)\in M$.} $\kappa$ where $\lambda>\kappa$ is a measurable cardinal, it is consistent that a strongly inaccessible cardinal $\kappa$ is $n$-$d$-stationary for all $n\in \omega$, but $\kappa$ is not weakly compact.
\end{theorem}

The organization of this paper is: 
\begin{enumerate}
    \item In Section \S\ref{section: prelim}, we record some preliminary facts regarding higher order stationary sets.
    \item In Section \S\ref{section: nonstronglimit}, we prove Theorem \ref{theorem: mainnonstronglimit0} and its 2-cardinal generalization.
    \item In Section \S\ref{section: PreparRadin} we prapare the ground model and present the relevant background for Radin forcing.
    \item In Section \S\ref{section: omegaThm}, we present a proof of Theorem \ref{theorem: mainstronglimit}.
    \item In Section \S\ref{section: questions}, we conclude with some open questions. 
\end{enumerate}
\subsection{Notations} Given a function $f:A\rightarrow B$ and $X\subseteq A$, the \textit{pointwise image} of $X$ by $f$ is the set $f''X:=\{f(x)\mid x\in X\}$. Given a set $X$ and a cardinal $\lambda$, we denote by   $P_\lambda X=\{Y\subseteq X\mid |Y|<\lambda\}$. For a sequence  $\l X_i\mid i<\lambda\r$ consisting of subsets of $\lambda$ we denote by the \textit{diagonal intersection} $\Delta_{i<\lambda}X_i:=\{\nu<\lambda\mid \forall\alpha<\nu, \ \nu\in X_i\}$. For a set of ordinals $A$, $\sup(A)=\cup A$ and we say that $A$ is \textit{bounded} in $\lambda$ if $\sup(A\cap\lambda)<\lambda$. We say that $A$ is \textit{closed} if it is closed in the order topology of the ordinals. A set $C$ is a \textit{club} at $\lambda$ is it is closed and unbounded, and the \textit{club filter} is
$$Cub_\lambda:=\{X\in P(\lambda)\mid \exists C\text{ a club }C\subseteq X\}$$
A set $S$ is called stationary in $\lambda$ is $S\cap C\neq\emptyset$ for every club $C$ in $\lambda$. We assume familiarity with forcing theory and refer the reader to \cite{CummingsHand} for background and standard notations. An elementary embedding $j$ is always a function $j:V\rightarrow M$ where $M$ is a transitive model, $crit(j)$ denoted the minimal ordinal which is moved by $j$. If $U$ is a $\sigma$-complete ultrafilter then $j_U:V\rightarrow M_U$ denoted the ultrapower by $U$. Given two finite sequence $\l x_1,...,x_n\r$ and $\l y_1,...,y_n\r$ we denote by $\l x_1,...,x_n\r^{\smallfrown}\l y_1,...,y_n\r=\l x_1,...,x_n,y_1,...,y_n\r$.


\section{Some preliminary facts}\label{section: prelim}

\subsection{$n$-stationarity, $(n,\chi)$-$s$-stationarity and $n$-$d$-stationarity}

Let us start with a useful lemma regarding the trace operation.
\begin{lemma}\label{lemma: traceOftrace} 
Fix a regular cardinal $\lambda$, $T\subset \lambda$, $\chi<\lambda$ and $n<\lambda$. \begin{enumerate}
    \item If $A\subset Tr_k(T)$ where $k\leq n$, then $\mathrm{Tr}_n(A)\subset \mathrm{Tr}_k(T)$,
    \item If $A\subset Tr^\chi_{k}(T)$ where $k \leq  n$, then $Tr_n^\chi(A)\subset Tr_{k}^\chi(T)$,
    \item If $A\subset Tr^d_{k}(T)$ where $k \leq  n$, then $Tr_n^d(A)\subset Tr_{k}^d(T)$. Furthermore, if $A\subset \Delta_{i<\lambda} Tr^d_k(T_i)$, then $Tr^d_n(A)\subset \Delta_{i<\lambda} Tr^d_k(T_i)$.
\end{enumerate}
\end{lemma}
\begin{proof}
For $(1)$, fix $\beta\in \mathrm{Tr}_n(A)$. Let $S\subset \beta$ be $m$-stationary for some $m<k$. We need to find $\beta'\in T\cap \beta$ such that $S\cap \beta'$ is $m$-stationary. Since $A\cap \beta$ is $n$-stationary, we can find $\beta_0\in A\cap \beta$ such that $S\cap \beta_0$ is $m$-stationary. As $A\subset \mathrm{Tr}_k(T)$, $T\cap \beta_0$ is $k$-stationary. Therefore, there is $\beta'\in T\cap \beta_0$ such that $S\cap \beta'$ is $m$-stationary.

For $(2)$, fix $\beta\in \mathrm{Tr}_n^\chi(A)$. Let $\langle S_i: i<\chi'\rangle$ for some $\chi'<\chi$ and $S_i\subset \beta$ being $(k_i, \chi)$-$s$-stationary where $k_i<k$ be given. Since $A\cap \beta$ is $(n,\chi)$-$s$-stationary, there exists $\beta'\in A\cap \beta$ such that $\beta'\in \bigcap_{i<\chi'} Tr^{\chi}_{k_i}(S_i)$. As $T\cap \beta'$ is $(k,\chi)$-$s$-stationary, there exists $\beta^*\in T\cap \beta'$ such that $\beta^*\in \bigcap_{i<\chi'} Tr^\chi_{k_i}(S_i\cap \beta')=\bigcap_{i<\chi'} Tr^{\chi}_{k_i}(S_i)$. In other words, $T\cap \beta$ is $(k,\chi)$-$s$-stationary.

For $(3)$, fix $\beta\in \mathrm{Tr}_n^d(A)$. Let $\langle S_i: i<\beta\rangle$ where each $S_i\subset \beta$ is $k_i$-$d$-stationary where for some $k_i<k$. Since $A\cap \beta$ is $n$-$d$-stationary, there exists $\beta'\in A\cap \beta$ such that $\beta'\in \bigcap_{i<\beta'} Tr^{d}_{k_i}(S_i)$. As $T\cap \beta'$ is $k$-$d$-stationary, there exists $\beta^*\in T\cap \beta'$ such that $\beta^*\in \bigcap_{i<\beta^*} Tr^d_{k_i}(S_i)$. In other words, $T\cap \beta$ is $k$-$d$-stationary. To see the ``furthermore" part, fix $\alpha\in Tr^d_n(A)$ and $i<\alpha$, we know that $A-(i+1)\subset Tr^d_k(T_i)$. By the previous argument, we have that $\alpha\in Tr_n^d(A-(i+1))\subset Tr^d_k(T_i)$.
\end{proof}
The combinatorial properties of the $n$-stationary sets are best expressed in the language of ideals and filters. Ideals are the standard absractization of the notion of ``smallness". Recall that a set $I\subseteq P(X)$ is an \textit{ideal} on $X$ if $\emptyset\in I$, $I$ is downward closed with respect to ``$\subseteq$" and closed under finite unions. We say that an ideal $I$ is \textit{proper} if $X\notin I$. The dual notion of an ideal is a \textit{filter}, i.e. given an ideal  $I$ we define the \textit{dual filter} $I^*:=\{X\setminus N\mid N\in I\}$. We extend the definition of $I^*$ to any set $I\subseteq P(X)$.  The set of \textit{positive sets} with respect to some ideal $I$ is denoted by $I^+:=P(X)\setminus I$. For more information about ideal and filters we refer the reader to \cite[Ch. 7]{Jech2003}
\begin{definition}
For every $n<\lambda$, let 
$NS^n_\lambda$, $NS^{(n,\chi)}_\lambda$, and $NS^d_\lambda$ be the set of all non-$n$-stationary, non-$(n,\chi)$-s-stationary, non-$n$-$d$-stationary subsets of $\lambda$ (resp.), and let $Cub^n_\lambda$, $Cub^{(n,\chi)}_\lambda$, $Cub^d_\lambda$ be the corresponding dual filters. 
\end{definition}

\begin{fact} \label{fact: NSfacts}
\begin{enumerate}
    \item If $T\notin NS^n_\lambda$, then $Tr_n(T)\in Cub^{m}_\lambda$ for any $m>n$. Indeed, $\lambda\setminus Tr_n(T)$ is not $m$-stationary as witnessed by the $n$-stationary set $T$. 
    \item Conversely, if $NS^m_\lambda$ is proper, then for every set $C\in Cub^{m}_\lambda$ there is an $n$-stationary set $T$ for some $n<m$ such that $Tr_n(T)\subseteq C$. To see this, since $\lambda\setminus C\in NS^{m}_\lambda$, there is some $n<m$ and a $n$-stationary set $T$ such that $Tr_n(T)\cap (\lambda\setminus C)=\emptyset$, namely, $Tr_n(T)\subseteq C$.
    \item We have that $NS^n_\lambda\subseteq NS^{m}_\lambda$ (and therefore $Cub^n_\lambda\subseteq Cub^{m}_\lambda$) for any $n\leq m < \lambda$. This follows from the fact that whenever $S$ is $m$-stationary, by Definition \ref{definition: higherorderstat}, it is also $n$-stationary.
    \item $NS^n_\lambda$ is always upward closed with respect to $\subseteq$. Indeed $\emptyset\in NS^0_\lambda\subseteq NS^n_\lambda$, if $X\in NS^n_\lambda$ and $Y\subseteq X$, then and $m$-stationary set $S$ for $m<n$ which witnesses that $X$ is not $n$-stationary will also witness that $Y$ is not $n$-stationary.
    \item If $NS^n_\lambda$ is an ideal, then it is proper iff $\lambda$ is an $n$-stationary cardinal.
\end{enumerate} 
\end{fact}

\begin{lemma}\label{lemma: previous}
Let $\lambda$ be regular and $n<\lambda$. Fix any $S\subset \lambda$. Then $S$ is $n+1$-stationary iff $\lambda$ is $n$-stationary and for any $n$-stationary $T\subset \lambda$, $Tr_n(T)\cap S\neq \emptyset$.
\end{lemma}

\begin{proof}
    We prove the non-trivial direction $(\leftarrow)$. Given any $m\leq n$ and $m$-stationary $W$, we need to show $Tr_m(W)\cap S\neq \emptyset$. If $m=n$, then we are done by the hypothesis. So assume $m<n$. Let $T=Tr_m(W)$. Since $Cub^{m+1}_\lambda\subset Cub^n_{\lambda}$ by Fact \ref{fact: NSfacts}, we have that $T$ is $n$-stationary. By the hypothesis, $Tr_n(T)\cap S\neq\emptyset$. Fix $\beta\in Tr_n(T)\cap S$. Our goal is to show that $\beta\in Tr_m(W)$, namely $W\cap \beta$ is $m$-stationary. Let $V$ be a $k$-stationary subset of $\beta$ for some $k<m$. Since $T\cap \beta$ is $n$-stationary, there is some $\beta'\in T$ such that $V\cap \beta'$ is $k$-stationary. Recall that $T=Tr_m(W)$. Then we have that $W\cap \beta'$ is $m$-stationary. As a result, there is $\beta''\in W\cap \beta'$ such that $V\cap \beta''$ is $k$-stationary. Hence, we have found $\beta''\in Tr_k(V)\cap W\cap \beta$.
\end{proof}


We record the following fact for the other ideals. The proof is similar to that of Lemma \ref{lemma: previous}.

\begin{lemma}\label{lemma: previous other ideals}
Let $\lambda$ be regular and $n, \chi<\lambda$ with $\chi$ infinite. Fix any $S\subset \lambda$. Then 
\begin{enumerate}
    \item $S$ is $(n+1, \chi)$-$s$-stationary iff $\lambda$ is $(n,\chi)$-$s$-stationary and for any $(n,\chi)$-$s$-stationary sets $\langle T_i: i<\chi'\rangle$ for some $\chi'<\chi$, $\bigcap_{i<\chi'} Tr_n^\chi(T_i)\cap S\neq \emptyset$.
      \item $S$ is $n+1$-$d$-stationary iff $\lambda$ is $n$-$d$-stationary and for any $n$-$d$-stationary sets $\langle T_i: i<\lambda\rangle$, $\Delta_{i<\lambda} Tr_n^d(T_i)\cap S\neq \emptyset$.
\end{enumerate}

\begin{proof}
    We only prove the non-trivial direction ($\Leftarrow$) in the following. 
    \begin{enumerate}
        \item Given $\langle T_i: i<\chi'\rangle$ such that each $T_i$ is $(k_i,\chi)$-$s$-stationary for some $k_i\leq n$, we note that if $k_i<n$, then $Tr^\chi_{k_i}(T_i)$ is $(n,\chi)$-$s$-stationary. To see this, suppose $\vec{S}=\langle S_k: k<\chi''<\chi\rangle$ is given with each $S_k$ being $(m_k,\chi)$-$s$-stationary for some $m_k<n$. Since $\lambda$ is $(n,\chi)$-$s$-stationary, there is some $\alpha\in \bigcap_{k<\chi''} Tr^\chi_{m_k}(S_k) \cap Tr^\chi_{k_i}(T_i)$.
        For each $i<\chi'$, let $T'_i = T_i$ if $k_i=n$ and $T'_i=Tr^\chi_{k_i}(T_i)$ if $k_i<n$. Apply the hypothesis, we have that $S\cap \bigcap_{i<\chi'} Tr^\chi_{n}(T'_i)\neq \emptyset$. Fix some $\beta$ in the intersection. Then $\beta\in S\cap \bigcap_{i<\chi'} Tr^\chi_{k_i}(T_i)$ by Lemma \ref{lemma: traceOftrace} (2).
        \item The proof is similar to the previous one, except that we apply Lemma \ref{lemma: traceOftrace}(3) instead. 
    \end{enumerate}
\end{proof}

\end{lemma}

        \begin{lemma}\label{lemma: (xi,chi)-s-non-stationary ideal}
            \begin{enumerate}
                \item $NS^{(\xi,\chi)}_\kappa$ is a proper subset of $P(\kappa)$ iff $\kappa$ is $(\xi,\chi)$-s-stationary.
                \item $NS^{(\xi,\chi)}_\kappa$ is always (might be $P(\kappa)$) $\chi$-complete ideal when $\chi$ is an infinite regular cardinal.
                \item  Suppose that $\kappa$ is $(n,\chi)$-$s$-stationary such that $\chi$ is an infinite cardinal and either $n$ is a successor ordinal or $\cf(n)\geq \chi$. For any $C\in Cub^{(n,\chi)}_\kappa$, there is $R$ which is $(k,\chi)$-$s$-stationary for some $k<n$ such that $Tr^\chi_{k}(R)\subseteq C$. 
            \end{enumerate}  
        \end{lemma}
        \begin{proof}
            \begin{enumerate}
                \item Immediate. 
                \item Closure under subsets is immediate. To see that it is $\chi$-complete, suppose that $\l A_i\mid i<\chi'<\chi\r\subseteq NS^{(\xi,\chi)}_\kappa$, then for each $i<\chi'$ there is a sequence $\l T_{j,i}\mid j<\chi'_i\r$ with $\chi'_i<\chi$ and $\eta_j^i<\xi$ such that each $T_i$ is $(\eta_j^i,\chi)$-$s$-stationary and 
                $A_i \cap \bigcap_{j<\chi_i'} Tr^\chi_{\eta_j^i}(T_{j,i})=\emptyset$. It is clear that $\langle T_{j,i}: j<\chi_i', i<\chi'\rangle$ witnesses that $\bigcup_{i<\chi'} A_i \in NS^{(\xi,\chi)}_\kappa$.

                \item By definition, there is a sequence $\l T_i\mid i<\chi'<\chi\r$ with each $T_i$ being $(\eta_i,\chi)$-$s$-stationary for some $\eta_i<n$ such that $R:=\cap_{i<\chi'}Tr^\chi_{\eta_i}(T_i)\subseteq C$. By the hypothesis about $n$, we can find some $k<n$ such that $\eta_i\leq k$ for all $i<\chi'$.
                
                We claim that $R$ is $(k,\chi)$-$s$-stationary. Let $\l S_j\mid j<\chi''\r$ be such that $\chi''<\chi$ each $S_j$ is $(m_j,\chi)$-$s$-stationary for some $m_j<k$. Apply the fact that $\kappa$ is $(n,\chi)$-$s$-stationary to the sequence $\l S_j\mid j<\chi''\r^{\smallfrown}\l T_i\mid i<\chi'\r$ to conclude that $R\cap \bigcap_{j<\chi''} Tr^\chi_{m_j}(S_j)\neq \emptyset$. Finally note that $Tr^\chi_k(R)\subseteq \cap_{i<\chi'}Tr^\chi_{\eta_i}(T_i)\subseteq C$, by Lemma \ref{lemma: traceOftrace} (2).
            \end{enumerate}
        \end{proof}

    \begin{lemma}\label{lemma: d-ideal}
            \begin{enumerate}
                \item $NS^{\xi,d}_\kappa$ is a proper subset of $P(\kappa)$ iff $\kappa$ is $\xi$-$d$-stationary.
                \item $NS^{\xi,d}_\kappa$ is always a (might be $P(\kappa)$) normal ideal.
                \item  Suppose that $\kappa$ is $n+1$-$d$-stationary. For any $C\in Cub^{n+1,d}_\kappa$, there is $R$ which is $ n$-$d$-stationary such that $Tr^d_{n}(R)\subseteq C$. 
            \end{enumerate}  
        \end{lemma}

         \begin{proof}
            \begin{enumerate}
                \item Immediate. 
                \item Closure under subsets is immediate. To see that it is normal, suppose that $\l A_i\mid i<\kappa\r\subseteq NS^{\xi,d}_\kappa$, then for each $i<\kappa$ there is a sequence $\l T_{j,i}\mid j< \kappa\r$ with $\eta_j^i<\xi$ such that each $T_i$ is $\eta_j^i$-$d$-stationary and 
                $A_i \cap \Delta_{j<\kappa} Tr^d_{\eta_j^i}(T_{j,i})=\emptyset$. Fix some bijection $g$ from $\kappa\times \kappa$ to $\kappa$ such that on a club $D\subset \kappa$, for any $\alpha\in D$, for any $i,j<\alpha$, $g(j,i)<\alpha$. Let $T'_{g(j,i)}=T_{j,i}$.
                As a result, the sequence $\{D\}\cup\l Tr^{d}_{\eta^i_j}(T'_{g(j,i)}): j,i<\kappa\rangle$ witnesses that $\bigtriangledown_{i<\kappa} A_i\in NS^{\xi,d}_\kappa$. The reason is that $\Delta_{j,i<\kappa} T'_{g(j,i)} \cap C =_{def} \{\alpha\in C: \forall j,i<\alpha, \alpha\in Tr^d_{\eta^i_j}(T_{j,i})\}$ avoids $\bigtriangledown_{i<\kappa} A_i=_{def}\{\beta: \exists i<\beta, \beta\in A_i\}$.

                \item By definition, there is a sequence $\l T_i\mid i<\kappa\r$ with each $T_i$ being $n$-$d$-stationary such that $R:=\Delta_{i<\kappa}Tr^d_{n}(T_i)\subseteq C$. 
                That $R$ is $n$-$d$-stationary follows from the fact that $\kappa$ is $n$-$d$-stationary. Finally note that $Tr^d_n(R)\subseteq \Delta_{i<\kappa}Tr^d_{n}(T_i)\subseteq C$, by Lemma \ref{lemma: traceOftrace} (3).
            \end{enumerate}
        \end{proof}

\end{definition}

\subsection{The relationship between different ideals in the constructible universe}

\begin{lemma}\label{lemma: normaln-stationary}
    Fix a cardinal $\kappa$ and $n<\kappa$.
    If $NS^n_\kappa$ is normal, and for any $k<n$, $\{\alpha<\kappa: NS^{k}_\alpha = NS^{k,d}_\alpha\}\in Cub^{n}_\kappa \cap Cub^{n,d}_\kappa$ and $NS^{k}_\kappa = NS^{k,d}_\kappa$, then $NS^{n}_\kappa = NS^{n,d}_\kappa$.
\end{lemma}

Let us clarify that ``$\{\alpha<\kappa: NS^{k}_\alpha = NS^{k,d}_\alpha\}\in Cub^{n}_\kappa \cap Cub^{n,d}_\kappa$" really means whenever $NS^n_\kappa$ (respectively $NS^{n,d}_\kappa$) is proper, then $\{\alpha<\kappa: NS^{k}_\alpha = NS^{k,d}_\alpha\}\in Cub^{n}_\kappa$ ($Cub^{n,d}_\kappa$).

\begin{proof}
    First suppose $NS^n_\kappa$ is not proper. In particular, this is the case when $\kappa$ is singular as $NS^n_\kappa$ is assumed to be normal. We need to show $NS^{n,d}_\kappa$ is also not proper. If $NS^k_{\kappa}$ is not proper for some $k<n$, then by the hypothesis, $NS^{k,d}_\kappa$ is not proper, which in turn implies that $NS^{n,d}_\kappa$ is not proper. Hence, we may assume $NS^{k}_\kappa$ is proper for all $k<n$. By the assumption, there is some $k$-stationary $T\subset \kappa$ such that $Tr_k(T)=\emptyset$. Suppose for the sake of contradiction that $NS^{n,d}_{\kappa}$ is proper. 
    As $T$ is $k$-$d$-stationary by the hypothesis, there is some $\alpha$ such that $T\cap \alpha$ is $k$-$d$-stationary and $NS^k_\alpha = NS^{k,d}_\alpha$. In particular, $T\cap \alpha$ is $k$-stationary. This contradicts with the fact that $Tr_k(T)=\emptyset$.

    We may now assume that $NS^{n}_\kappa$ is proper. Let $B= \{\alpha<\kappa: \forall k<n, NS^{k}_\kappa = NS^{k,d}_\kappa\}$, then $B\in Cub^{n}_\kappa  \cap Cub^{n,d}_\kappa$, as $NS^n_\kappa$ and $NS^{n,d}_\kappa$ are normal and in particular $\kappa$-complete.
    
    First we show that if $A$ is $n$-stationary, then $A$ is $n$-$d$-stationary. In particular, this implies $NS^{n,d}_\kappa$ is proper. Let $\langle T_i: i<\kappa\rangle$ be given such that each $T_i$ is $k_i$-$d$-stationary for some $k_i<n$. By the hypothesis, we know that $T_i$ is $k_i$-stationary. Since $NS^n_\kappa$ is normal, there is some $\alpha\in \Delta_{i<\kappa} Tr_{k_i}(T_i) \cap A \cap B$. We check that $\alpha\in \Delta_{i<\kappa} Tr^d_{k_i}(T_i)$. Fix $i<\alpha$, $T_i\cap \alpha$ is $k_i$-stationary. As $\alpha\in B$, $T_i\cap \alpha$ is $k_i$-$d$-stationary.

    Next we show that if $A$ is $n$-$d$-stationary, then $A$ is $n$-stationary. Let $T$ be a $k$-stationary subset of $\kappa$ for some $k<n$. By the hypothesis, $T$ is $k$-$d$-stationary. Find $\alpha\in A\cap Tr_{k}^d(T) \cap B$. Then $T\cap \alpha$ is $k$-$d$-stationary. Since $\alpha\in B$, $T\cap \alpha$ is $k$-stationary.
\end{proof}

\begin{proposition}\label{proposition: propagate}
    Suppose for any cardinal $\kappa$ and any $n<\kappa$, $NS^n_\kappa$ is normal, then for all $\kappa$ and $n<\kappa$, $NS^n_\kappa = NS^{n,d}_\kappa$.
\end{proposition}

\begin{proof}
    Suppose otherwise for the sake of contradiction. Fix the least cardinal $\kappa$ and then the least $n<\kappa$ such that $NS^{n}_\kappa\neq NS^{n,d}_\kappa$. Note that $\kappa$ is regular and $n>1$. We will reach a contradiction by verifying that the hypotheses of Lemma \ref{lemma: normaln-stationary} are satisfied. For $k<n$,
    \begin{itemize}
        \item $NS^n_\kappa$ is normal by the assumption,
        \item  $NS^k_\kappa = NS^{k,d}_\kappa$ by the minimality of $n$ and
        \item  $\{\alpha<\kappa: NS^k_\alpha = NS^{k,d}_\alpha\} = \kappa$ by the minimality of $\kappa$. 
    \end{itemize}
\end{proof}

Similar proofs to those in Lemma \ref{lemma: normaln-stationary} and Proposition \ref{proposition: propagate} give the following: 

\begin{lemma}\label{lemma: chi-complete-n-stationary}
    Fix a cardinal $\kappa$, $n<\kappa$ and $\chi<\kappa$.
    If $NS^n_\kappa$ is $\chi$-complete, and for any $k<n$, $\{\alpha<\kappa: NS^{k}_\alpha = NS^{(k,\chi)}_\alpha\}\in Cub^{n}_\kappa \cap Cub^{(n,\chi)}_\kappa$ and $NS^{k}_\kappa = NS^{(k,\chi)}_\kappa$, then $NS^{n}_\kappa = NS^{(n,\chi)}_\kappa$.
\end{lemma}

\begin{proposition}\label{proposition: chipropagate}
    Suppose for any cardinal $\kappa$, any $n<\kappa$ and $\chi<\kappa$, $NS^n_\kappa$ is $\chi$-complete, then for all $\kappa$ and $n<\kappa$, $NS^n_\kappa = NS^{(n,\chi)}_\kappa$.
\end{proposition}

\begin{corollary}\label{corollary: equivalence}
    If $V=L$, then all the following are equivalent for $\xi<\kappa$: 
    \begin{enumerate}
        \item $\kappa$ is $\Pi^1_{\xi}$-indescribable,
        \item $\kappa$ is $\xi+1$-stationary,
        \item $\kappa$ is $(\xi+1, \chi)$-$s$-stationary for some (any) $\chi<\kappa$,
        \item $\kappa$ is $\xi+1$-$d$-stationary.
    \end{enumerate} 
    In fact, in the theorem above, all the ideals corresponding to each clause are the same. Namely, $\mathbf{\Pi}^1_\xi\cap P(\kappa) = NS^{\xi+1}_\kappa= NS^{(\xi+1,\chi)}_\kappa = NS^{\xi+1,d}_{\kappa}$. 
\end{corollary}

\begin{proof}
This follows from \cite[Corollary 2.5]{BagariaMagidorSakai}, \cite[Theorem 5.1]{Bagaria} and Propositions \ref{proposition: propagate}, \ref{proposition: chipropagate}. 
\end{proof}

\subsection{2-cardinal higher order stationarity}
Sakai \cite{SakaiSlides} generalized the higher order stationarity notions to the two-cardinal setting.

\begin{definition}[Sakai \cite{SakaiSlides}]\label{definition: 2-cardinal}
For a regular cardinal $\kappa$, a set $A\supset \kappa$ and $n\in \kappa$, 
\begin{itemize}
\item $S\subset P_\kappa A$ is $0$-stationary if $S$ is $\subset$-cofinal\footnote{Namely, for every $X\in P_\kappa A$ there is $Y\in S$ such that $X\subseteq Y$.} in $P_\kappa A$,
\item $S$ is $n$-stationary if for any $m<n$, any $m$-stationary $T\subset P_\kappa A$, there is $B\in S$ such that 
	\begin{itemize}
	\item $\mu = B\cap \kappa$ is a regular cardinal,
	\item $T\cap P_\mu B$ is $m$-stationary.
	\end{itemize}
	The collection of $B$ satisfying the above is called the \emph{$m$-trace} of $T$, written as $\mathrm{Tr}_m(T)$ (this is slight abuse of notation but there should be no difficulty inferring from the context).
\item $P_\kappa A$ is $n$-stationary if $P_\kappa A$ is $n$-stationary as a subset of $P_\kappa A$. 
\end{itemize}
\end{definition}
\begin{remark}
In the original definition, only $n\in \omega$ was considered. Generalizing that to $n\in \kappa$ poses no difficulty.
\end{remark}

\begin{remark}\label{remark: reduce}
If $T\subset P_\kappa A$ and $B\in \mathrm{Tr}_m(T)$, then for any $T'\subset T$ with $T\cap P_\kappa B \subset T'$, we have $B\in \mathrm{Tr}_m(T')$. The reason is that $T'\cap P_\kappa B = T\cap P_\kappa B $.
\end{remark}

\begin{lemma}\label{lemma: traceoftrace2}
If $A\subset \mathrm{Tr}_m(T)$, then $\mathrm{Tr}_m(A)\subset \mathrm{Tr}_m(T)$.
\end{lemma}

\begin{proof}
Let $B\in \mathrm{Tr}_m(A)$ with $B\cap \kappa=\mu$ regular. Let $k<m$ and $S\subset P_\mu B$ be a $k$-stationary subset. We need to show that $T\cap \mathrm{Tr}_k(S)\neq \emptyset$. Since $A\cap P_\mu B$ is $m$-stationary, there is $C\in A$ such that $C\cap \kappa=\nu$ is regular and $S\cap P_\nu C$ is $k$-stationary. Since $C\in \mathrm{Tr}_m(T)$, $T\cap P_\nu C$ is $m$-stationary. Therefore, there is some $D\in T\cap P_\nu C$ such that $D\cap \nu=\delta$ is a regular cardinal and $S\cap P_\nu C \cap P_\delta D = S\cap P_\delta D$ is $k$-stationary in $P_\delta D$.
\end{proof}

\section{Higher order stationary reflection at a non strong limit cardinal}\label{section: nonstronglimit}

\begin{definition}
    Let $\lambda$ be a regular cardinal and $\kappa$ be a cardinal. An ideal $I$ on $\lambda$ is 
    \begin{enumerate}
        \item \emph{uniform} if $[\lambda]^{<\lambda}\subset I$,
        \item \emph{normal} if for any $\langle A_i: i<\lambda\rangle\in [I]^{\lambda}$, the diagonal union $\bigtriangledown_{i<\kappa} A_i =_{def} \{\alpha: \exists i<\alpha, \alpha\in A_i\}\in I$
        \item \emph{$\kappa$-saturated} if for any $\langle B_j: j<\kappa\rangle\in [I^+]^\kappa$, there exist $j_0\neq j_1<\kappa$ such that $B_{j_0}\cap B_{j_1}\in I^+$.
        \end{enumerate}
\end{definition}

\begin{fact}\label{fact: generate}
    Let $I$ be uniform normal $\kappa$-saturated ideal on $\lambda$ where $\kappa<\lambda$. Let $G\subset P(\lambda)/I$ be generic over $V$. Then in $V[G]$,
    \begin{enumerate}
        \item \cite[Chapter 2]{Foreman2010ideals} 
         there is an elementary embedding $j: V\to M\simeq \Ult(V,G)$ such that $crit(j)=\lambda$, $V[G]\models {}^\lambda M\subset M$, and 
        \item \cite[Theorem 17.1]{kanamori1994} the ideal $\bar{I}$ generated by $I$ is uniform normal and $\kappa$-saturated.
        \item for any $\dot{X}$ such that $\Vdash_{P(\lambda)/I} \dot{X}\in \bar{I}^*$, there exists $X\in I^*$ such that $\Vdash_{P(\lambda)/I} X\subset \dot{X}$ (this follows from the fact that $P(\lambda)/I$ is $\kappa$-c.c and $I$ is $\kappa$-complete).
        \item a set $A$ is in $I$ if and only if $\Vdash_{P(\lambda)/I}\lambda\notin \lusim{j}(\check{A})$. 
    \end{enumerate}
\end{fact}

\begin{theorem}\label{theorem: nonstronglimit}
Let $\lambda$ be a regular cardinal carrying a uniform normal $\kappa$-saturated ideal $I$ for some $\kappa<\lambda$. Fix also some $\chi<\lambda$. Then for all $k<\lambda$, $NS^{k}_\lambda$, $NS^{(k, \chi)}_\lambda$, $NS^{k, d}_\lambda$ are all proper ideals on $\lambda$. In particular, $\lambda$ is $\lambda$-stationary, $(\lambda,\chi)$-$s$-stationary and $\lambda$-$d$-stationary.
\end{theorem}

\begin{proof}
We prove the following statement $(*)_n$ by induction on $n\in \mathrm{Ord}$. For any $\lambda>\kappa,n$ such that $\lambda$ carries a uniform normal $\kappa$-saturated ideal $I$, for any $T\subset \lambda$, 
    \begin{enumerate}
        \item if $T$ is $n$-stationary, then $Tr_n(T)\in I^*$.
        \item if $T$ is $(n,\chi)$-$s$-stationary, then $Tr_n^\chi(T)\in I^*$, and 
        \item if $T$ is $n$-$d$-stationary, then $Tr_n^d(T)\in I^*$.
    \end{enumerate}

\begin{claim}
    $(*)_n$ for all $n\in \mathrm{Ord}$ implies that:  $\lambda$ is $\lambda$-stationary, $(\lambda,\chi)$-$s$-stationary and $\lambda$-$d$-stationary whenever $\lambda>\chi$ carries a uniform normal $\kappa$-saturated ideal $I$ for some $\kappa<\lambda$.
\end{claim}

\begin{proof}[Proof of the Claim]
    \begin{enumerate}
        \item Since $\lambda$ is $\lambda$-stationary iff $\lambda$ is $k$-stationary for all $k<\lambda$, the first clause is immediate.
        \item Given $\chi'<\chi$ and $\langle T_i: i<\chi'\rangle$ such that each $T_i$ is $(\eta_i, \chi)$-$s$-stationary for some $\eta_i<\lambda$, by the hypothesis we know there are $A_i\in I^*$ such that $A_i\subset Tr^\chi_{\eta_i}(T_i)$ for each $i<\chi'$. Since $I^*$ is $\lambda$-complete, we have that $\bigcap_{i<\chi'}Tr^\chi_{\eta_i}(T_i) \supset \bigcap_{i<\chi'} A_i \in I^*$.
        \item Given $\langle T_i: i<\lambda\rangle$ such that each $T_i$ is $\eta_i$-$d$-stationary for some $\eta_i<\lambda$, by the hypothesis, there are $A_i\in I^*$ such that $A_i \subset Tr^d_{\eta_i}(T_i)$ for each $i<\lambda$. Since $I$ is normal, $ \Delta_{i<\lambda} Tr^d_{\eta_i}(T_i)\supset \Delta_{i<\lambda} A_i\in I^*$.
        \end{enumerate}
\end{proof}

Base case $n=1$. Recall that in this case, 
\begin{itemize}
    \item $T$ is $1$-stationary if $T$ is a stationary subset of $\lambda$,
    \item $T$ is $(1,\chi)$-$s$-stationary if $T \cap \lambda\cap \cof(\geq \chi)$ is stationary in $\lambda$,
    \item $T$ is $1$-$d$-stationary if $T$ is a stationary subset of $\lambda$.
\end{itemize}

Let $G\subset P(\lambda)/I$ be generic over $V$. If $T$ is a stationary subset of $\lambda$ (stationary relative to $\lambda\cap \cof(\geq \chi)$), then $\Vdash_{P(\lambda)/I}$ $T\subset \lambda$ is stationary since the forcing satisfies $\kappa$-c.c. Note that $j(T)\cap\kappa=T$ where $j:V\rightarrow M$ is the elementary embedding from fact \ref{fact: generate} item (1). In particular, $M\models \kappa \in j(Tr_1(T))$ and by fact \ref{fact: generate} (4) $Tr_1(T)\in I^*$ ($Tr^\chi_1(T)\in I^*$, or $Tr^d_1(T)\in I^*$).

Suppose we have proved $(*)_i$ for all $i < n$, let us show $(*)_{n}$. Let $G\subset P(\lambda)/I$ be generic over $V$ and let $j: V\to M\simeq \Ult(V,G)$ be an ultrapower embedding in $V[G]$. Let us first assume $T\subset \lambda$ is $n$-stationary. It suffices to show that $M\models T$ is $n$-stationary, as the conclusion follows from the elementarity of $j$. Suppose for the sake of contradiction that $M\models T$ is not $n$-stationary. Since $V[G]\models {}^\lambda M\subset M$ by Fact \ref{fact: generate} (1), $V[G]\models T$ is not $n$-stationary. As a result, there exists a $k$-stationary $S\subset \lambda$ such that $Tr^{V[G]}_k(S)\cap T =\emptyset$ for some $k<n$. By Fact \ref{fact: generate} (2), the ideal $\bar{I}$ generated by $I$ is uniform normal and $\kappa$-saturated in $V[G]$. Therefore, we can apply the induction hypothesis $(*)_k$ in $V[G]$ to conclude that there exists $C\in \bar{I}^*$ such that $C\subset Tr^{V[G]}_k(S)$. Since $\bar{I}$ is generated by $I$, we may assume that $C\in I^*$. In particular, $C\in V$. Apply $(*)_k$ in $V$, we know that $C$ is $k$-stationary. As a result, $Tr_k(C)\cap T \neq \emptyset$. Fix $\alpha\in Tr_k(C)\cap T$. Apply $j$ to see that $\alpha\in j(Tr_k(C)\cap T)=Tr^{V[G]}_k(C)\cap T$. Hence $$M\models \alpha\in Tr_k(C)\cap T \subset Tr_k(Tr_k(S))\cap T \subset Tr_k(S)\cap T$$ by Lemma \ref{lemma: traceOftrace}. Contradicting the fact that $Tr^{V[G]}_k(S)\cap T =\emptyset$ in $V[G]$.

As the proof for the case where $T$ is $(n,\chi)$-$s$-stationary and $n$-$d$-stationary is similar to the above, we only sketch the differences. Let us assume $T$ is $n$-$d$-stationary for concreteness. Proceed as above and assume $M\models T$ is not $n$-$d$-stationary for the sake of contradiction. There exists $\langle S_i: i<\lambda\rangle$ such that each $S_i$ is $\eta_i$-$d$-stationary for some $\eta_i<n$, and $\Delta_{i<\lambda} (Tr^d_{\eta_i}(S_i))^{V[G]}\cap T =\emptyset$. In $V[G]$, apply the induction hypothesis and the normality of $\bar{I}$, for each $k<n$, we can get $C_k\in I^*$ such that $$C_k\subset \Delta_{i<\lambda}^k (Tr^d_{\eta_i}(S_i))^{V[G]}=_{def} \{\alpha: \forall i<\alpha, \text{ if }\eta_i=k, \text{ then }\alpha\in (Tr^d_k(S_i))^{V[G]}\}.$$
By Fact \ref{fact: generate} (3), we may assume that $\langle C_k: k<n\rangle\in V$. For each $k<n$, the induction hypothesis $(*)_k$ in $V$ implies that $C_k$ is $k$-$d$-stationary. As a result, we can find $\alpha\in T\cap \bigcap_{k<n} Tr^d_k(C_k)$. Apply $j$, we know that in $M$, $\alpha\in T\cap Tr^d_k(C_k)$ for each $k<n$. We check that $\alpha\in \Delta_{i<\lambda} (Tr^d_{\eta_i} (S_i))^{V[G]}$, which gives the desired contradiction. Fix $i<\alpha$ and $\eta_i=k$. By the definition of $C_k$, we have that $V[G]\models C_k-(i+1) \subset Tr^d_{k}(S_i)$. By Lemma \ref{lemma: traceOftrace}, $V[G]\models Tr^d_{k}(C_k-(i+1)) \subset Tr^d_{k}(S_i)$. But then $V[G]\models \alpha\in Tr^d_{k}(C_k-(i+1)) \subset Tr^d_{k}(S_i)$.
\end{proof}

\begin{proof}[Proof of Theorem \ref{theorem: mainnonstronglimit0}]
Let $\lambda$ be a measurable cardinal and a $\kappa$-c.c forcing $P$ be given where $\kappa<\lambda$. A theorem of Kunen (\cite[Lemma 2]{Kunen1978}) gives that in $V^P$, there exists a $\lambda$-complete normal $\kappa$-saturated ideal on $\lambda$. Then we apply Theorem \ref{theorem: nonstronglimit} to get the conclusion as desired.
\end{proof}

Let us turn our attension to the 2-cardinal higher order stationary reflection principles (see Definition \ref{definition: 2-cardinal}). 
Sakai \cite{SakaiSlides} posed the following question: 
    For $n\geq 3$, is it consistent that there is a cardinal $\kappa\leq 2^\omega$ such that $P_\kappa(\lambda)$ is $n$-stationary for all $\lambda\geq \kappa$? 
    
We answer this question positively in the following, adapting the proof of Theorem \ref{theorem: nonstronglimit} to the 2-cardinal setting. Recall that a cardinal $\kappa$ is called \textit{$\lambda$-supercompact} if there is a fine normal measure $U$ over $P_\kappa(\lambda)$. Equivalently, if there is an elementary embedding $j:V\rightarrow M$ such that $crit(j)=\kappa$, $M$ is transitive and $M^\lambda\subseteq M$. $\kappa$ is called \textit{supercomapct} if it is $\lambda$-supercompact for every $\lambda$. The following is a 2-cardinal version of the classic Mitchell order on normal ultrafilters on measurable cardinals.

\begin{definition}
We define an ordinal function $o$ on the set $\mathcal{NF}$ of all normal fine measures on $P_\kappa \lambda$ recursively as follows: for all $U\in \mathcal{NF}$,
	\begin{enumerate}
	\item $o(U)\geq 0$, 
	\item $o(U)\geq \xi$ if for any $\xi'<\xi$, there exists some $W\in \Ult(V, U)\cap \mathcal{NF}$ such that $o(W)\geq \xi'$.
	\end{enumerate}
For $U\in \mathcal{NF}$, $o(U)=\xi$ iff $o(U)\geq \xi$ but $o(U)\not\geq \xi+1$.
\end{definition}

\begin{observation}
If $\kappa$ is a supercompact cardinal, then for any $\lambda\geq \kappa$ and $\xi<\lambda$, there is a normal fine measure $U$ on $P_\kappa \lambda$ such that  $o(U)\geq \xi$.
\end{observation}

\begin{proof}
    Recall that for a normal fine ultrafilter $U$ on $P_\kappa\lambda$ and $\eta$, $o(U)\geq \eta$ if for any $\xi<\eta$, there exists some normal fine ultrafilter $W$ on $P_\kappa\lambda$ with $o(W)\geq \xi$ which belongs to $M_U\simeq \Ult(V,U)$.

    Fix $\lambda\geq \kappa$. Suppose for the sake of contradiction that $$\sup \{o(U)+1: U \text{ is a normal fine ultrafilter on }P_\kappa\lambda\}=\eta<\lambda.$$

    Let $j: V\to M$ witness that $\kappa$ is $\lambda^+$-supercompact. Let $W$ be the normal fine ultrafilter on $P_\kappa \lambda$ derived from $j$. Let $i: V\to N\simeq \Ult(V, W)$ and let $k: N\to M$ be defined such that $k([f]_W)=j(f)(j''\lambda)$. The following facts are standard (see \cite{kanamori1994}):
    \begin{enumerate}
        \item   both $i$ and $k$ are elementary and $j=k\circ i$, 
        \item  ${}^{\lambda^{<\kappa}} N\subset N$ and $crit(k)\geq (2^{\lambda^{<\kappa}})^{+N}$, 
        \item $W\not\in N$.
    \end{enumerate} 
  As a result, by elementarity, $$N\models \sup \{o(U)+1: U \text{ is a normal fine ultrafilter on }P_\kappa\lambda\}=\eta<\lambda.$$
    Note that since $N$ is sufficiently closed, any normal fine ultrafilter on $P_\kappa\lambda$ in $N$ is a normal fine ultrafilter on $P_\kappa\lambda$ in $V$. But then by the definition, $o(W)\geq \eta$, which is a contradiction. 
\end{proof}

\begin{theorem}
Let $\kappa$ be a supercompact cardinal and $P$ be a forcing satisfying $\nu$-c.c for some $\nu<\kappa$. In $V^P$, $P_\kappa \lambda$ is $n$-stationary for any $n<\kappa$ and $\lambda\geq \kappa$. 
\end{theorem}

\begin{proof}
The proof is similar to before, so we only highlight the modifications.
We prove the following statement $(\star)_n$ by induction on $n\in \mathrm{Ord}$: for any $\nu$, any $\lambda>\kappa> \max\{n, \nu\}$, any forcing $P$ satisfying $\nu$-c.c and any normal fine measure $U$ on $P_\kappa \lambda$ such that $o(U)\geq n$, the following holds in $V^{P}$: 
for any $n$-stationary $T\subset P_\kappa\lambda$, there exists $A\in U$ such that $Tr_{n}(T)\supset A$. As $(\star)_1$ is easy to be seen to hold, let us focus on the inductive step.

Suppose we have proved $(\star)_i$ for all $i<n$ and let us show $(\star)_{k}$. Let $\lambda>\kappa>\{n,\nu\}$, forcing $P$ and normal fine ultrafilter $U$ on $P_\kappa\lambda$ with $o(U)\geq n$ be given.
Let $j: V\to M\simeq \mathrm{Ult}(V, U)$ be the supercompact ultrapower embedding, in particular, ${}^\lambda M \subset M$. Let $G\subset P$ be generic over $V$ and $T\subset P_\kappa\lambda\in V[G]$ be $n$-stationary. We can continue to force and find $G^*\subset j(P)$ generic over $V$ extending $j''G$, such that we can lift $j$ to $j^+: V[G]\to M[G^*]$. In $V[G]$, standard arguments show that the ideal generated by the dual of $U$ is the same as $\{X\subset P_\kappa \lambda: \  \Vdash_{j(P)/j''G} j''\lambda\not\in \dot{j}^+(X)\}$. It suffices to show that ${j^+}'' T$  remains $n$-stationary in $P_\kappa j''\lambda$ in $M[G^*]$. Then we finish by the elementarity of $j$. 

Suppose for the sake of contradiction that $j^+{}'' T$ is not $n$-stationary in $P_\kappa j''\lambda$. 
In $M[G^*]$, let $S\subset P_\kappa j''\lambda$ be some $m$-stationary set such that $Tr_m(S)\cap j^+{}'' T =\emptyset$ where $m<n$. Since $o(U)\geq n$, we can find some normal fine ultrafilter $W\in M$ on $P_\kappa\lambda$ such that $o(W)\geq m$. We may identify $W$ as a normal fine ultrafilter $W'$ on $P_\kappa j''\lambda$ induced by $j\restriction \lambda\in M$.

Applying the induction hypothesis $(\star)_m$ in $M$ with respect to $W'$, $S$ and $j(P)$, we know that there is $B\in W'$ such that $M[G^*]\models B\subset Tr_{m}(S)$. Since in $P_\kappa \lambda$ and $P_\kappa j''\lambda$ are isomorphic, we know that $B'=j^{-1} (B)=\{j^{-1}(a): a\in B\}$ is in $W$. In particular, both $B, B'$ are in $V$. As a result, applying $(\star)_l$ for all $l\leq m$ in $V$, we get that $B'$ is an $m$-stationary subset of $P_\kappa\lambda$ in $V[G]$. In $V[G]$, let $D\in T\cap Tr_{m}(B')$ as $T$ is $n$-stationary.

Let $\mu=D\cap \kappa$ and we know $\mu$ is a regular cardinal in $\kappa$. By the elementarity of $j^+$, we have $j^+(D)={j^+}'' D \in {j^+}''T\cap (Tr_{m}(j(B')))^{M[G^*]}$. Note that ${j^+}''D\in  (Tr_{m}({j}''B'))^{M[G^*]}$, since ${j}'' V$ contains $(P_\mu {j^+}'' D) \cap j(B')$ (Remark \ref{remark: reduce}). To see this, let $a\in (P_\mu {j^+}'' D) \cap j(B')$, since $B'\in V$, we have that $j(B')\in M$. Thus $a\in M$. As a result, $a'=j^{-1}(a)\in V$. Hence, in $V[G]$, we must have that $a'\subset D\cap B'$
and $j(a')=j'' a' = a \in  {j}'' V$. Therefore, in $M[G^*]$, ${j^+}''D\in  Tr_{m}({j}''B')=Tr_{m}(B)\subset Tr_{m}(S)$ by Lemma \ref{lemma: traceoftrace2}. This contradicts with the fact that $Tr_m(S)\cap j^+{}'' T =\emptyset$ in $M[G^*]$.
\end{proof}

\section{Preparing the ground model and Radin forcing}\label{section: PreparRadin}
Start with a model of GCH where $\kappa$ is an $H(\lambda^{++})$-hypermeasurable cardinal where $\lambda$ is the least  measurable cardinal greater than $\kappa$. Our goal is to produce a universe $V$ where $2^\kappa=\lambda^+$ and there exists an elementary embedding $j:V\to M$ such that
    \begin{enumerate}
        \item $H(\lambda^{++})\subset M$,
        \item for every $X\subset \lambda$, there is $g\in V$ such that $j(g)(\kappa)=X$,
        \item for any $n\in \omega$, $\lambda$ is $(n,\kappa^+)$-$s$-stationary.
    \end{enumerate}

Let $r:\kappa\rightarrow\kappa$ be the function that takes any $\alpha$ to the minimal measurable cardinals $\alpha<r(\alpha)$. Since $\kappa$ is an $H(\lambda^{++})$-hypermeasurable cardinal, $r:\kappa\rightarrow\kappa$. Since the preparation is standard, we will only sketch the proof and refer the readers to the relevant literature for more details. Specifically, we follow largely \cite{CummingsGCH}, \cite{Honzik} and \cite{CummingsStrong}.

We will use some standard facts about term-space forcing. 

\begin{definition}
Let $P$ be a forcing and $\dot{Q}$ be a $P$-name for a forcing. Define $\dot{Q}/P$ to be the poset consisting of terms $\dot{\sigma}$ such that $\Vdash_{P} \dot{\sigma}\in \dot{Q}$. The order on $\dot{Q}/P$ is: $\dot{\sigma}\leq \dot{\tau}$ iff $\Vdash_P \dot{\sigma}\leq \dot{\tau}$.
\end{definition}

\begin{fact}\label{fact: term}
\begin{enumerate}
    \item \cite[Proposition 22.3]{CummingsHand} Fix a forcing $P$ and a $P$-name for a forcing $\dot{Q}$. Let $G\subset P$ be generic over $V$ and let $H\subset \dot{Q}/P$ be generic over $V$. Then $I=\{i_G(\dot{\tau}): \dot{\tau}\in H\}$ is an $i_G(\dot{Q})$-generic filter over $V[G]$.
    \item \cite[Fact 2]{CummingsGCH} Let $\kappa$ be such that $\kappa^{<\kappa}=\kappa$ and $P$ be a $\kappa$-c.c forcing. Let $\dot{Q}$ be a $P$-name for $\mathrm{Add}(\kappa,\gamma)$. Then in $V$, $\mathrm{Add}(\kappa,\gamma)$ is forcing equivalent to $\dot{Q}/P$.
\end{enumerate}

\end{fact}

\subsection{Step One}
The first stage is to ensure that there is a universe $V_0$ in which 	
	\begin{enumerate}
	\item $\kappa$ is an $H(\lambda^{++})$-hypermeasurable cardinal,
	\item GCH holds at all inaccessible $\alpha\leq\kappa$ and $\beta\geq \lambda$,
	\item there is an elementary embedding $j: V_0\to M$ such that 
            \begin{itemize}
                \item $H(\lambda^{++})\subset M$,
                \item $crit(j)=\kappa$,
                \item $j(r)(\kappa)=\lambda$, and
                \item ${}^\kappa M\subset M$
            \end{itemize}
  along with $i: V_0\to N$ being the ultrapower by the normal measure derived from $j$, there is $F\in V_0$ that is generic for $i(Add(\kappa, \lambda^+))$ over $N$.
	\item $V_0$ is a $\kappa^{++}$-c.c forcing extension of $V$.
	\end{enumerate}

 For the construction, see \cite[Corollary 2.7]{Honzik}. Apter and Cummings \cite{ApterCummings} independently, in some unpublished work, has an alternative way of achieving the above.

\subsection{Step Two}

We may take $V_0$ from the previous subsection as our ground model in this subsection. The second step is to perform the Easton support iteration $\langle P_\beta, \dot{Q}_\alpha : \alpha\leq \kappa, \beta\leq\kappa+1\rangle$ such that for any $\alpha<\kappa$, $\dot{Q}_\alpha$ is trivial unless $\Vdash_{P_\alpha}\alpha$ is inaccessible, in which case $\dot{Q}_\alpha$ is a $P_\alpha$-name for $\mathrm{Add}(\alpha,r(\alpha)^+)$. Finally, let $\dot{Q}_\kappa$ be the $P_\kappa$-name for $\mathrm{Add}(\kappa,\lambda^+)$. Let $G=G_\kappa*g_\kappa$ be $V$-generic for $P_{\kappa+1}=P_\kappa*\dot{Q}_\kappa$. Let $j:V\rightarrow M$ be the embedding from Step One. We may without loss of generality assume that $j=j_E$ where $E$ is a short $(\kappa,\lambda^{++})$-extender on $\kappa$. Let $i: V\to N$ be the ultrapower by the normal ultrafilter on $\kappa$ derived from $j$.

\begin{proposition}\label{proposition: prepare}
    In $V[G]$ we can lift $j\subseteq j^+:V[G]\rightarrow M[H]$ such that:
    \begin{enumerate}
        \item $(H(\lambda^{++}))^{V[G]}\subseteq M[H]$.
        \item For every $\xi<\lambda^+$ there is $g\in V[G]$ such that $j^+(g)(\kappa)=\xi$.
        \item $2^\kappa=2^\lambda=\lambda^+$.
    \end{enumerate}
\end{proposition}
\begin{proof}
    We need to construct a generic for $j(P_{\kappa+1})$. By definition of $P_{\kappa+1}$ and by elementarity of $j$, we have that $$j(P_{\kappa+1})=(P_\kappa*\mathrm{Add}(\kappa, \lambda^+)*P_{(\kappa,j(\kappa))}*\mathrm{Add}(j(\kappa),j(\lambda)^+))^M$$

where $P_{(\kappa,j(\kappa))}^M$ is an iteration starting at the first $M^{P_{\kappa+1}}$-inaccessible above $\kappa$ (and in particular above $\lambda^+$). In particular, after forcing $j(P_{\kappa+1})$,  $(2^\lambda)^{M^{j(P_{\kappa+1})}}$ remains $\lambda^+$-- this explains $(3)$. Up to $\kappa+1$ we can take $G$ as the $M$-generic filter.
In $V[G]$, we can find some $G_{(\kappa, j(\kappa))}$ that is $M[G]$-generic for $P_{(\kappa,j(\kappa))}^{M[G]}$ (see \cite[Proposition 15.1 and the paragraph before Lemma 25.5]{CummingsHand}). The key point is that we can find $g_{j(\kappa)}\in V[G]$ that is generic for $\mathrm{Add}(j(\kappa),j(\lambda)^+)^{M[G*H_0]}$. This uses crucially (3) in the preparation of Step One. Let us outline some key points. Let $i: V\to N$ be the ultrapower by the normal ultrafilter derived from $j$ and let $k: N\to M$ be the natural map defined as $k([f])=j(f)(\kappa)$. Then by Fact \ref{fact: term} (2), $i(\dot{\mathrm{Add}}(\kappa, \lambda^+))/P_{i(\kappa)}$ is forcing equivalent to $i(Add(\kappa,\lambda^+))$. Hence we can find in $V$ a generic for $i(\dot{\mathrm{Add}}(\kappa, \lambda^+))/P_{i(\kappa)}$ over $N$. By Fact \ref{fact: term} (1), this can be transferred along the embedding $k$ to a generic for $j(\mathrm{Add}(\kappa,\lambda^+))$ over $M[G_{j(\kappa)}]$.  More details can be found in \cite[Theorem 2.11]{Honzik} or \cite[The second step, Page 245-246]{CummingsStrong}.

We procceed with the usual Woodin surgery argument \cite[Chapter 25]{CummingsHand} and alter the values of $g_{j(\kappa)}$ (we abuse notation and keep denoting the altered functions by $g_{j(\kappa),\alpha}$) so that for every $\xi<\lambda^+$, $g_{j(\kappa),j(\xi)}\restriction \kappa=g_{\kappa,\xi}$ and $g_{j(\kappa),j(\xi)}(\kappa)=\xi$.  It is routine to check that $g_{j(\kappa)}$ is still generic, and that $j''G\subseteq G*G_{(\kappa,j(\kappa)}*g_{j(\kappa)}=:H$. So we may lift in $j\subseteq j^+:V[G]\rightarrow M[H]$. Note that we have $(H(\lambda^{++}))^{V[G]}\subset M[H]$. To see this, as $P_{\kappa+1} \in H(\lambda^{++})^V\subset M$ is a $\kappa^+$-c.c forcing extension of $V$, we know that $H(\lambda^{++})^{V[G]}\subset M[G] \subset M[H]$.
Since we made sure that $j^+(g_{\kappa,\xi})(\kappa)=g_{j(\kappa),j(\xi)}(\kappa)=\xi$, we have shown that $(1)-(2)$ hold.
\end{proof}

\begin{corollary}
    
    Assume $GCH$, $\kappa$ is a $H(\lambda^{++})$-hypermeasurable cardinal where $\kappa<\lambda$ is a measurable cardinal. Then there is a generic extension $V^*$ where:
    \begin{enumerate}
        \item $2^\kappa=2^\lambda=\lambda^+>\lambda$.
        \item there is an elementary embedding $j^*:V^*\rightarrow M^*$ such that:
        \begin{enumerate}
            \item $crit(j)=\kappa$, $H((2^\kappa)^+)^{V^*}\subseteq M^*$.
            \item for every $X\subseteq \lambda$ there is $g\in V^*$ such that $j(g)(\kappa)=X$.
        \end{enumerate}
        \item For all $n<\kappa$, $\lambda$ is $(n,\kappa^+)$-$s$-stationary.
    \end{enumerate} 
\end{corollary}
\begin{proof}
    Let $V^*$ and $j^*:V^*\rightarrow M^*$ be as in the conclusion of Proposition \ref{proposition: prepare}. In particular $(1),(2a)$ hold. Moreover, for every $\xi<\lambda^+$, there is a function $g\in V^*$ such that $j^*(g)(\kappa)=\xi$.  To see $(2b)$, take any $X\subseteq \lambda$ and factor $j^*$ through the ultrapower embedding $i^*:V^*\rightarrow N^*$ by the normal ultrapower derived from $j^*$ and let $k:N^*\rightarrow M^*$ be the factor map such that $k\circ i^*=j^*$ defined by $k(i^*(g)(\kappa))=j^*(g)(\kappa)$. Since $\lambda^++1\subseteq Im(k)$, we have that $crit(k)>\lambda^+$. In particular, for every $X\subseteq \lambda$, such that $X\in N^*$, $k(X)=X$. Also note that $k(P^{N^*}(\lambda))=P^{M^*}(\lambda)=P^{N^*}(\lambda)$. The reason is that $N^*\models 2^\lambda=\lambda^+$ and $crit(k)>\lambda^+$.
    Therefore, every $X\subseteq \lambda$ is of the form $j^*(f)(\kappa)$. Finally, to see $(3)$, note that $V^*$ is a generic extension of $V$ by a $\kappa^{++}$-c.c forcing (the Step One forcing is $\kappa^{++}$-c.c and the Step Two forcing is $\kappa^+$-c.c), so we may reason as in the proof of Theorem \ref{theorem: mainnonstronglimit0} following the proof of Theorem \ref{theorem: nonstronglimit}.  
    \end{proof}

From now on, let us denote by $V^*=V$ and $j^*=j$ the model and the elementary embedding of the previous corollary. Let us give a brief description of the notations we use for Radin forcing  \cite{Radin} and relevant background for our main result. For full detailed definitions and proofs consult \cite{Gitik2010}.
\subsection{Radin Forcing}
 Let $\bar{U}$ be a measure sequence derived from $j$ of length $\lambda$ i.e. $\bar{U}=\l\kappa\r^{\smallfrown}\l U(\xi)\mid \xi<\lambda\r$ such that for every $\xi<\lambda$ we define recursively $$U(\xi)=\{X\subseteq V_{\kappa}\mid \bar{U}\restriction \xi\in j(X)\}$$ and $U(0)$ is just the normal measure derived from $j$ using $\kappa$. More generally, a \textit{measure sequence} is any sequence of ultrafilter $\bar{w}$ of any length, denoted by $lh(\bar{w})$, which is derived from some elementary embedding $j_{\bar{w}}$ in the same way $\bar{U}$ was derived, always starting with the seed $crit(j)$ which we denote by $\kappa(\bar{w})=crit(j_{\bar{w}})$. We denote by 
 $\mathcal{MS}$ the class of all measure sequences. It is well-known (see for example \cite{Gitik2010}) that we may only consider $\bar{w}\in \mathcal{MS}$ which concentrate on $\mathcal{MS}$ i.e $\mathcal{MS}\cap V_{\kappa(\bar{w})}\in \bigcap_{\xi<lh(\bar{w})}\bar{w}(\xi)=: \bigcap\bar{w}$. 
 \begin{remark}
     We will always assume that given a set $A\in\bigcap \bar{w}$, for every $\bar{v}\in A$, $A\cap V_{\kappa(\bar{v})}\in\bigcap\bar{v}$. Such a $\bar{v}$ is said to be \textit{addable} to $A$. For generally, we say that $\vec{\eta}=(\bar{v}_1,...,\bar{v}_n)$ is addable to $(\bar{w},A)$, and denote it by $\vec{\eta} << A$, if for every $1\leq i\leq n$, $\bar{v}_i\in A$ and $A\cap V_{\kappa(\bar{v}_i)}\in \bigcap \bar{v}_i$.
 \end{remark}
 Given two measure sequences $\bar{u},\bar{v}$, we denote $\bar{u}\prec \bar{w}$ if $\bar{u}\in V_{\kappa(\bar{w})}$. If $B\subseteq \mathcal{MS}$,
 then $$B^{<\omega}=\bigcup_{n<\omega}\{\l\bar{u}_1,\dots,\bar{u}_n\r\in B^n\mid \bar{u}_1\prec\bar{u}_2\prec\dots\prec\bar{u}_n\}$$
 Let us follow the description of Radin forcing from \cite{Gitik2010} with the exception that $q\leq p$ means $q$ is a stronger condition.
 \begin{definition}
     The \textit{Radin forcing with $\bar{U}$}, denoted by $R_{\bar{U}}$ consist of all finite sequences $p=\l d_i\mid i\leq k\r$ such that each $d_i$ is either an ordinal $\kappa_i<\kappa$ (which we identify as the measure sequence $\bar{u}^p_i=\l\kappa_i\r$) or a pair $(\bar{u}^p_i,A_i^p)$, such that:
     \begin{enumerate}
        \item $d_k=(\bar{U},A^p)$ where $A^p\in\bigcap\bar{U}$.
         \item $\bar{u}^p_1\prec\bar{u}^p_2\prec\dots\prec\bar{u}^p_k$.
         \item If $d_{i}$ is a pair then $A^p_{i}\in\bigcap \bar{u}^p_i$ and if $i>0$ then for every $\bar{u}\in A^p_i$, $\bar{u}^p_{i-1}\prec \bar{u}$.
     \end{enumerate}
 \end{definition}
 \begin{notation}
     We denote the \textit{ length} of the condition $lh(p)=k$, the \textit{lower part} $p_0=\l d_i\mid i<lh(p)\r$, the \textit{upper part} $(\bar{U},A)$ (so we may write $p=p_0^{\frown}(\bar{U},A)$), Let $\kappa_0(p)=\kappa(d_{k-1})$, and $R_{<\kappa}=\{p_0\mid p\in R_{\bar{U}}\}$. 
 \end{notation}
 \begin{definition}
Let $p,q\in R_{\bar{u}}$. We say that $p$ is a \textit{direct extension} of $q$ and $p\leq^* q$ if:
\begin{enumerate}
    \item $lh(p)=lh(q)$.
    \item For every $i<lh(p)$, $\bar{u}^p_i=\bar{u}^q_i$ and $A^p_i\subseteq A^q_i$.
\end{enumerate}
 \end{definition}
 \begin{definition}
     Let $p\in R_{\bar{u}}$. A \textit{one-step extension} of $p$ is  obtained by choosing $i\leq lh(p)$ $\bar{u}^p_i$ and $\bar{v}\in A^p_i$ which is addable to $(\bar{u}^p_i,A^p_i)$ and forming the condition
     $$p\cat{\bar{v}}:=\l d_j\mid j<i\r^{\smallfrown}(\bar{v},A^p_i\cap V_{\kappa(\bar{v}})^{\smallfrown}(\bar{u}^p_i,A^p_i\setminus V_{\kappa(\bar{v})+1})^\smallfrown\l d_j\mid i<j\leq lh(p)\r.$$
     We define recursively, $p\cat{\bar{v}_1,\dots,\bar{v}_{n+1}}=(p\cat{\bar{v}_1,\dots,\bar{v}_n})\cat{\bar{v}_{n+1}}$.
 \end{definition}

\begin{remark}
We clarify the notations regarding concatenation and one-step extension. 
\begin{itemize}
\item $p^\frown (\bar{w}, A)$ is the string concatenation in the usual sense. In particular, in order for this to make sense, $\kappa(\bar{w})>\kappa(\bar{v})$ for all $\bar{v}\in p$ and we always need to specify the large set $A$.
\item $p\cat{\bar{w}}$ is the one step extension  as in the definition above. In particular, $\kappa(\bar{w})<\kappa(\bar{v})$ for some $\bar{v}\in p$ and we only specify the measure sequence since the measure one set for $\bar{w}$ is already determined.
\end{itemize}

\end{remark} 
 
 We define the order by $p\leq q$ if there are $\bar{v}_1,\dots\bar{v}_n$, such that $\bar{v}_i\in A^q_i$ and $p\leq^*q\cat{\bar{v}_1,\dots\bar{v}_n}$.
 Let us list some basic properties of $R_{\bar{U}}$ (for the proof see \cite{Gitik2010}):
 \begin{proposition}
     \begin{enumerate}
         \item $R_{\bar{U}}$ is $\kappa^+$-c.c.
         \item For any condition $p$, $(R_{\bar{U}}/p,\leq^*)$ is $\kappa(\bar{u}^p_0)$-directed closed. In particular, if $lh(p)=1$ the $(R_{\bar{U}}/p,\leq^*)$ is $\kappa$-closed.
         \item For any $p=\l d_i\mid i\leq lh(p)\r\in R_{\bar{U}}$, and any $i<lh(p)$, we can factor $$R_{\bar{U}}/p\simeq (R_{\bar{u}^p_i}/p\restriction i+1)\times (R_{\bar{U}}/p\setminus i+1)$$
         where $p\restriction i+1=\l d_j\mid j\leq i\r$ and $p\setminus i+1=\l d_j\mid i<j\leq lh(p)\r$.
         \item $R_{\bar{U}}$ satisfies the Prikry property: For any sentence in the forcing language $\sigma$, and any condition $p\in R_{\bar{U}}$ there is a direct extension $p^*\leq^* p$ such that $p^*\Vdash \sigma\vee p^*\Vdash \neg \sigma$\footnote{In this situation we say that $p^*$ decides $\sigma$ and denote $p^*||\sigma$.}.
         \item $R_{\bar{U}}$ satisfies the strong Prikry property: For every dense open set $D\subseteq R_{\bar{U}}$ and any condition $p$, there is  $p^*\leq^* p$ and a $p^*$-fat tree \footnote{Namely, a tree $T\subseteq (\mathcal{MS}\cap V_\kappa)^{n}$ for some $n<\omega$, such that for each $t\in T$, $p^*{}\cat{t}\in R_{\bar{U}}$ and there is $i\leq lh(p^*)$, $\xi<lh(\bar{u}^{p^*}_i)$ such that $succ_{T}(t):=\{\bar{w}\mid t^\smallfrown\bar{w}\in T\}\in \bar{u}^{p^*}_i(\xi)$.} $T$ such that for every maximal branch $t\in T$, $p^*{}\cat{t}\in D$.
     \end{enumerate}
 \end{proposition} 
 We will need the following proposition that reduces $R_{\bar{U}}$-names to names which depends on bounded information:
 \begin{proposition}\label{Prop: reducing to bounded names}
    Let $\l \gamma_\alpha\mid \alpha<\kappa\r\in V$ be any sequence of ordinals below $\kappa$ and $p=p_0^{\smallfrown}(\bar{U},A)\in\mathbb{R}_{\bar{U}}$ and $\l \dot{x}_\alpha\mid \gamma_\alpha<\kappa\r$ be a sequence of $\mathbb{R}_{\bar{U}}$-names such that $p\Vdash\dot{x}_\alpha\subseteq\gamma_\alpha$. Then there is $q\leq p$, $q_0=p_0$ and a function $f:A^{q}\rightarrow V_\kappa$ such that for every $\bar{w}\in A^{q}$, $f(\bar{w})$ is an $R_{\bar{w}}$-name forced by $q_0^{\smallfrown}(\bar{w},A^{q}\cap V_{\kappa(\bar{w})})$ to be a subset of $\kappa(\bar{w})$ and $q^{\smallfrown}\bar{w}\Vdash f(\bar{w})=\dot{x}_{\kappa(\bar{w})}$. 
\end{proposition}
\begin{proof}
    The proof is exactly as in \cite[Lemma 2.13]{OmerJing} exploiting the $\kappa$-closure of $\leq^*$ of the upper part to determine $\dot{T}\cap \gamma_{\kappa(\bar{w})}$.
\end{proof}
 \begin{definition}
    Let $G\subseteq R_{\bar{U}}$ be $V$-generic. We denote by
    $$\mathcal{MS}_G=\{\bar{u}\in \mathcal{MS}\mid \exists p\in G\ \exists i<lh(p)\ \bar{u}=\bar{u}^p_i\}$$
    The \textit{generic Radin club} is the set $O(\mathcal{MS}_G)=\{\alpha<\kappa: \exists \bar{u}\in \mathcal{MS}_G, \alpha=\kappa(\bar{u})\}$.
\end{definition}
We say that a set $A$ is \textit{generated} by a set in the ground model if there is $B\in V$ such that $A=O(B\cap\mathcal{MS}_G)$.
 Other useful lemmas concerning Radin forcing can also be found in \cite{Omer1} and \cite{OmerJing}.

\subsection{Compactness and stationarity in Radin Extensions}

Let $G\subseteq R_{\bar{U}}$ be $V$-generic. 
It turns out that some large cardinal properties of $\kappa$ in the generic extension $V[G]$ correspond to  combinatorial properties of $lh(\bar{U})$ (see the discussion in the introduction for some examples). One which is relevant to us the that of a weakly compact cardinal which is due to the second author and Ben-Neria:
\begin{lemma}[{\cite[Lemma 3.14]{OmerJing}}]\label{Lemma: not weakly compact}
    Suppose that $(2^\kappa)^M$ does not divide $lh(\bar{U})$, then in $V^{R_{\bar{U}}}\models \kappa$ is not weakly compact. 
\end{lemma}
In particular, the measure sequence $\bar{U}$ which we have prepared satisfies that $lh(\bar{U})=\lambda<(2^\kappa)^M$. To show that $\kappa$ has some reflection properties in the generic extension, we will need to analyze stationary sets and higher-order stationary sets in Radin forcing extensions. 

Ben-Neria \cite{Omer1} has characterized clubs and stationary sets in the generic extension using measure sequences:  
\begin{theorem}[\cite{Omer1}]\label{Thm: 0-Club} If $\bar{U}$ satisfies $cf(lh(\bar{U}))\geq\kappa^+$, where $\kappa=\kappa(\bar{U})$, then
given $p_0^{\smallfrown}(\bar{U},A)=p\Vdash  \dot{\tau}$ is a club subset of $\kappa$, there exists a measure one set
$A'\subseteq A$ and a set $\Gamma$ such that for some $\eta<\kappa$, $\Gamma\in U(\xi)$ for all $\xi\in[\eta,lh(\bar{U}))$ and  $p_0^{\smallfrown}(\bar{U},A')\Vdash O(\Gamma\cap \mathcal{MS}_G) \subset \dot{\tau}$.
\end{theorem}
It follows that if $cf(lh(\bar{U}))\geq\kappa^+$, any set $A\in V$ such that $A\in U(\xi)$ for unboundedly many $\xi$'s below $lh(\bar{U})$, will generate a stationary set in $V^{R_{\bar{U}}}$ i.e. $O(A\cap\mathcal{MS}_G)$ will intersect any club (see \cite[Proposition 15]{Omer1}). There are stationary sets which are not generated from a ground model set (\cite[Proposition 2.12]{OmerJing}) and the exact characterization of stationary sets appears in \cite[Theorem 19]{Omer1} and uses the notion of measure function:
\begin{definition}
    A \textit{measure function} is a function $b:\mathcal{MS}\rightarrow V_\kappa$ such that for every $\bar{u}\in \mathcal{MS}$, $b(\bar{u})\in\bigcap\bar{u}$.
\end{definition}
\begin{theorem}[{\cite[Theorem 19]{Omer1}}]
      Suppose that $\bar{U}$ satisfies $cf(lh(\bar{U}))\geq\kappa^+$, where $\kappa=\kappa(\bar{U})$, and $\dot{S}$ is $R_{\bar{U}}$-name such that 
      $p\Vdash  \dot{S}$ is a stationary subset of $\kappa$. Then there is $ e=e_0^{\smallfrown}(\bar{U},B)\leq p$ and a measure function $b$ such that for every $\vec{\eta}\in B^{<\omega}$:
    $$Z_{e_0}\setminus \vec{\eta}:=\{\bar{w} \in Z_{e_0}\mid e_0^\frown (\bar{w}, b(\bar{w}))^\frown (\bar{U}, B-V_{\kappa(\bar{w})+1})\Vdash \kappa(\bar{w})\in \dot{T}$$$$ \text{ and } \vec{\eta}<< b(\bar{w})\}\in U(\xi)$$  for unboundedly many $\xi<lh(\bar{U})$.
 \end{theorem}
For the sake of convenience, let us denote $Z_{e_0}\backslash \emptyset$ as $Z_{e_0}$.
This provided the main ingredient in \cite{Omer1} to guarantee stationary reflection in the generic extension:
\begin{theorem}[\cite{Omer1}]
    Suppose that $\bar{U}$ is a measure sequence such that $cf(lh(\bar{U}))\geq\kappa^{++}$, then $V^{R_{\bar{U}}}\models$ every stationary set at $\kappa$ reflects and moreover every $\kappa$-sequence of stationary subsets of $\kappa$ reflects diagonally (or in our terminology, $\kappa$ is $2$-$d$-stationary).  
\end{theorem}
Our intention is to generalize this characterization to higher levels of stationarity and show that the measure sequence we produced in the preparation, guarantees that $\kappa$ is $k$-$d$-diagonal-stationary for all $k\in \omega$ in $V^{R_{\bar{U}}}$. 

\begin{definition}
    Given $A \subset V_\kappa$, let
$$Ind_{\bar{U}}(A)=\{\alpha<lh(\bar{U})\mid A\in U(\alpha)\}$$
\end{definition}

\begin{definition}
Fix $n<\kappa$ and a measure sequence $\bar{U}$ on $\kappa$ of length $\lambda$. Let $\chi<\lambda$. We say that $A\subset V_\kappa$ is:
\begin{enumerate}
    \item $\bar{U}$-$(n,\chi)$-$s$-club if $Ind_{\bar{U}}(A)\in Cub^{(n,\chi)}_\lambda$
    \item $\bar{U}$-$(n,\chi)$-$s$-stationary if $Ind_{\bar{U}}(A)\not\in NS^{(n,\chi)}_\lambda$.
    \item $\bar{U}$-$(n,\chi)$-$s$-null if $Ind_{\bar{U}}(A)\in NS^{(n,\chi)}_\lambda$.
\end{enumerate}
\end{definition}

\section{Higher order stationary reflection in the Radin extension}\label{section: omegaThm}

In this section we present the proof of Theorem \ref{theorem: mainstronglimit}. We restate the theorem for the convenience of the reader.
\begin{theorem*}
        Relative to the existence of a $H(\lambda^{++})$-hypermeasurable cardinal $\kappa$ where $\lambda>\kappa$ is a measurable cardinal, it is consistent that a strongly inaccessible cardinal $\kappa$ is $n$-$d$-stationary for all $n\in \omega$, but $\kappa$ is not weakly compact.
\end{theorem*}

 By the results from Section \ref{section: PreparRadin}, we may assume the following:
there is an elementary embedding $j:V\rightarrow M$ with critical point $\kappa$ such that 
\begin{enumerate}
    \item $V\models 2^\kappa=2^\lambda=\lambda^+>\lambda$,
    \item $H(\lambda^{++})\subset M$ and ${}^\kappa M \subset M$,
    \item $\lambda$ is $(\omega,\kappa^+)$-$s$-stationary.
    \item For every $X\subseteq \lambda$ there is $f$ such that $j(f)(\kappa)=X$.
\end{enumerate}
By the fact that $H(\lambda^{++})\subset M$, we can then derive a measure sequence $\bar{U}$ of length $\lambda$ from $j$ (see \cite[Lemma 5.1]{Gitik2010}).
In particular, for $\xi<\lambda$, $U(\xi)$ concentrates on $\bar{u}\in\mathcal{MS}\cap V_\kappa$ such that $\bar{u}$ is derived from an embedding $j':V\rightarrow M'$ such that $\kappa(\bar{u})$ is the critical point of $j'$ and:
\begin{enumerate}
    \item $2^{\kappa(\bar{u})}>lh(\bar{u})$, 
    \item for every $X\subseteq lh(\bar{u})$ there is $f$ such that $j'(f)(\kappa(\bar{u}))=X$.
\end{enumerate} 

The reflected objects here are the initial segments of $\bar{U}$, namely $\{\bar{U}\restriction \xi: \xi<lh(\bar{U})\}$. By our assumption, any proper initial segment of $\bar{U}$ belongs to $M$, hence, reflection is possible. Furthermore, for each such reflected measure sequence, one can form an elementary embedding that derives it. We refer the reader to \cite[Lemma 5.1]{Gitik2010} for this type of arguments.

By shrinking to a measure one set in $\cap\bar{U}$, we may  abuse our notations by assuming that each $\bar{u}\in \mathcal{MS}\cap V_\kappa$, there is an embedding $j'$ from which $\bar{u}$ is derived as above. We call these measure sequences \emph{good}. If in addition, $\xi<\lambda$ is $(\rho,\kappa^+)$-$s$-stationary for $\rho<\omega$, then we may assume $U(\xi)$ concentrates on good measure sequences $\bar{u}$ such that $lh(\bar{u})$ is $(\rho, \kappa(\bar{u})^+)$-$s$-stationary.

The argument is to prove inductively on $n\in \omega-\{0\}$ that:
for any good measure sequence $\bar{U}$ on $\kappa$ with $lh(\bar{U})=\lambda$,
the following sequence of propositions holds.
\begin{proposition}[$\phi_{0,n}$]
     If $V^{R_{\bar{U}}}\models\kappa$ is $n$-$d$-stationary, then $lh(\bar{U})=\lambda$ is $(n-1, \kappa^+)$-s-stationary.
\end{proposition}

If in addition $\lambda$ is $(n-1,\kappa^+)$-$s$-stationary, then the following propositions hold:

\begin{proposition}[$\phi_{1,n}$]\label{claim: characterization3-stationary}
In $V^{R_{\bar{U}}}$, $S$ is $n$-$d$-stationary iff $S\cap O(\Gamma \cap \mathcal{MS}_G)\neq \emptyset$ for all $\Gamma\subset \mathcal{MS}$ that is $\bar{U}$-$(n-1,\kappa^+)$-s-club.
\end{proposition}

\begin{proposition}[$\phi_{2,n}$]\label{nstationary witness}
Let $\dot{T}$ be a $R_{\bar{U}}$-name such that $p\Vdash \dot{T}\subseteq \kappa$ is $n$-d-stationary. Then there is $e=e_0^{\smallfrown}\l \bar{U},B\r \leq p$ and a measure function $b$ such that:
\begin{enumerate}
    \item $Z_{e_0}:=\{\bar{w}: e_0^\frown (\bar{w}, b(\bar{w}))^\frown (\bar{U}, B-V_{\kappa(\bar{w}+1)})\Vdash \kappa(\bar{w})\in \dot{T}\}$ is $\bar{U}$-$(n-1,\kappa^+)$-$s$-stationary, 
    \item  For every $\vec{\eta}\in B^{<\omega}$, $Z_{e_0}\setminus \vec{\eta}:=\{\bar{w} \in Z_{e_0}\mid \vec{\eta}<< b(\bar{w})\}$ is $\bar{U}$-$(n-1,\kappa^+)$-s-stationary, where $\vec{\eta} << b(\bar{w})$ means for any measure sequence $\bar{v}$ appearing in $\vec{\eta}$, it is the case that $\bar{v}\in b(\bar{w})$ and $b(\bar{w})\cap V_{\kappa(\bar{v})}\in \bigcap \bar{v}$, namely, $\bar{v}$ can be added below $(\bar{w}, b(\bar{w}))$.
\end{enumerate}
We call such ($e$,$b$) an \emph{$n$-$d$-stationary witness} for $\dot{T}$
\end{proposition}
\begin{proposition}[$\phi_{3,n}$]\label{EasyDirectionN-StationaryChar}
    If $\dot{T}$ is a $R_{\bar{U}}$-name such that an $n$-d-stationary witness $(e,b)$ for $\dot{T}$ exists, then $e\Vdash\dot{T}$ is $n$-d-stationary.
\end{proposition}

If in addition $\lambda$ is $(n,\kappa^+)$-$s$-stationary, then the following proposition holds: 

\begin{proposition}[$\phi_{4,n}$]\label{claim: trace}
In $V^{R_{\bar{U}}}$, for any sequence of $n$-d-stationary sets $\langle S_i: i<\kappa\rangle$, there exists a $\bar{U}$-$(n,\kappa^+)$-$s$-club subset $\Gamma^*\subset \mathcal{MS}$ in $V$ such that $O(\Gamma^*\cap \mathcal{MS}_G) \subset \Delta_{i<\kappa}\mathrm{Tr}^d_n(S_i)$. 
\end{proposition}




\begin{remark}
    Strictly speaking, we should decorate these propositions with $\bar{U}$, namely $\phi_{j,n}$ should be $\phi^{\bar{U}}_{j,n}$ for $j=0,1,2,3,4$, since for each $n\in \omega$, we quantify over all good measure sequences. In the following, we suppress the superscript if the measure sequence we are dealing with is $\bar{U}$. Otherwise, we will always decorate with the superscript to make precise which good measure sequence the induction hypothesis is applied to.
\end{remark}

Fix an embedding $j:V\rightarrow M$ with critical point $\kappa$ witnessing that $\bar{U}$ is good. Namely, 
\begin{enumerate}
    \item $\bar{U}$ is derived from $j$,
    \item $2^\kappa>lh(\bar{U})=:\lambda$,
    \item for every $X\subseteq \lambda$ there is $f$ such that $j(f)(\kappa)=X$.
\end{enumerate}

The base case $n=1$. Note that $\phi_{0,1}$ is saying that if $V^{R_{\bar{U}}} \models \kappa$ is regular, then $lh(\bar{U})=\lambda$ must have cofinality $\geq \kappa^+$. This is true and follows from the arguments in \cite[Lemma 5.11-5.13]{Gitik2010} and the fact if $\xi<2^\kappa$, then $\xi$ is not a weak repeat point for $\bar{U}$.
Recall that for an ordinal $\theta \in \cof(\geq \kappa^+)$, $(0,\kappa^+)$-$s$-stationary subsets of $\theta$ are just the unbounded subsets of $\theta$ and when $\theta$ is regular, $1$-$d$-stationary subsets of $\theta$ are just stationary subsets of $\theta$ as the club filter at $\theta$ is always normal.
Thus $\phi_{1,1},\phi_{2,1}, \phi_{3,1}, \phi_{4,1}$ were proved in \cite{Omer1}.

We focus on the inductive case $n>1$. The argument is, to some extent, a generalization of that in \cite{Omer1}. 

\begin{proof}[Proof of $\phi_{0,n}$]
If $lh(\bar{U})$ is not $(n-2,\kappa^+)$-$s$-stationary, then by $\phi_{0,n-1}$, in $V^{R_{\bar{U}}}$, $\kappa$ is not $n-1$-$d$-stationary
and in particular not $n$-$d$-stationary. So we may assume that $lh(\bar{U})$ is $(n-2,\kappa^+)$-$s$-stationary. Suppose $\lambda=lh(\bar{U})$ is not $(n-1, \kappa^+)$-$s$-stationary.
Let $\l A_i\mid i<\kappa\r$ be a sequence of $(n-2,\kappa^+)$-s-stationary subsets of $\lambda$ such that $\cap_{i<\kappa}Tr_{n-2}^{\kappa^+}(A_i)=\emptyset$.

    For each $i<\kappa$, by our assumption about the embedding $j$, let $f_i$ be such that $j(f_i)(\kappa)=A_i$ and let $\Gamma_i=_{def}\{\bar{w}\in \mathcal{MS}: lh(\bar{w})\in f_i(\kappa(\bar{w}))\}$ so that in particular $Ind_{\bar{U}}(\Gamma_i)=A_i$. Then by $\phi_{1,n-1}$, $O(\Gamma_i\cap \mathcal{MS}_G)$ is a $n-1$-$d$-stationary set for any $i<\kappa$. 
    
    We would like to prove that $\l O(\Gamma_i\cap MS_G)\mid i<\kappa\r$ witness that $\kappa$ is not $n$-$d$-stationary. Indeed $$B:=\{\bar{v}\mid \forall i<\kappa(\bar{v}), \ Ind_{\bar{v}}(\Gamma_i\cap V_{\kappa(\bar{v})})\text{ is $(n-2, \kappa(\bar{v})^+)$-$s$-stationary in }lh(\bar{v})\}$$ is $\bar{U}$-null. To see this, for every $\xi<lh(\bar{U})$, $Ind_{\bar{U}\restriction\xi}(\Gamma_i)=A_i\cap \xi$ and since $\cap_{i<\kappa}Tr_{n-2}^{\kappa^+}(A_i)=\emptyset$, $\bar{U}\restriction \xi\notin j(B)$. We may assume without loss of generality that $\mathcal{MS}_G\cap B=\emptyset$. If $\kappa(\bar{v})\in O(\mathcal{MS}_G)$ would have been a $n-1$-$d$-stationary point of all the $O(\Gamma_i\cap \MS_G)$ for $i<\kappa(\bar{v})$. In particular $\kappa(\bar{v})$ is $n-1$-$d$-stationary. By the induction hypothesis $\phi_{0, n-1}^{\bar{v}}$, $lh(\bar{v})$ is $(n-2,\kappa(\bar{v})^+)$-$s$-stationary. By $\phi_{2,n-1}^{\bar{v}}$,  $Ind_{\bar{v}}(\Gamma_i\cap V_{\kappa(\bar{v})})$ should have been $(n-2, \kappa(\bar{v})^+)$-$s$-stationary in $lh(\bar{v})$. This would mean that $\bar{v}\in B \cap \MS_G$, which is a contradiction.\end{proof}
   
From now on, assume $\lambda$ is $(n-1, \kappa^+)$-$s$-stationary.
    \begin{proof}[Proof of $\phi_{1,n}$]
Work in $V[G]$. 
\begin{itemize}
\item $(\leftarrow)$: Let $\l T_i\mid i<\kappa\r$ be a sequence of $n-1$-$d$-stationary sets. As $\lambda$ is $(n-1,\kappa^+)$-$s$-stationary, we can apply $\phi_{4,n-1}$ to find $\Gamma \in V$ which is a $\bar{U}$-$(n-1,\kappa^+)$-club such that $O(\Gamma\cap \MS_G)\subseteq \Delta_{i<\kappa}Tr^d_{n-1}(T_i)$. By our assumption on $S$, we have that $\emptyset\neq S\cap O(\Gamma\cap \mathcal{MS}_G)\subseteq S\cap \Delta_{i<\kappa}Tr_n^d(T_i)$. Hence $S$ is $n$-$d$-stationary.
\item $(\rightarrow)$: Suppose there exists a set $\Gamma\subset \mathcal{MS}$ that is $\bar{U}$-$(n-1,\kappa^+)$-$s$-club such that $S\cap O(\Gamma\cap \mathcal{MS}_G)=\emptyset$, we will cook up a $n-1$-$d$-stationary set $H$ such that $Tr^d_{n-1}(H)\cap S$ is bounded. This implies that $S$ is not $n$-$d$-stationary.\footnote{Note that if $S$ is $n$-$d$-stationary, and $H$ is $n-1$-$d$-stationary, then $Tr^d_{n-1}(H)\cap S$ must be unbounded.}
Since $\lambda$ is assumed to be $(n-1,\kappa^+)$-$s$-stationary, by Lemma \ref{lemma: (xi,chi)-s-non-stationary ideal} (3), we can find $T_0 \subset \lambda$ which is $(n-2,\kappa^+)$-$s$-stationary such that $Tr^{\kappa^+}_{n-2}(T_0)\subseteq C=Ind_{\bar{U}}(\Gamma)$. By our assumption on $j$, there is $f\in V$ such that $j(f)(\kappa)=T_0$.
 Let $$\Gamma'=\{\bar{w}\in V_{\kappa}\cap \mathcal{MS}: lh(\bar{w})\in f(\kappa(\bar{w}))\}.$$ Note that $Ind_{\bar{U}}(\Gamma')=T_0$. Indeed, $\bar{U}\restriction \xi\in j(\Gamma')$ iff $\xi\in j(f)(\kappa)=T_0$. Let $H=O(\Gamma'\cap \mathcal{MS}_G)$. Since $\Gamma'$ is $\bar{U}-(n-2,\kappa^+)$-$s$-stationary, by $\phi_{1,n-1}$, $H$ is an $n-1$-$d$-stationary subset of $\kappa$ in $V[G]$. Note that:
 \begin{claim}
     $Tr^d_{n-1}(H)=O(\Gamma^*\cap \mathcal{MS}_G)$ where 
 $$\Gamma^*=\{\bar{u}:  \Gamma'\cap V_{\kappa(\bar{u})}\text{ is }\bar{u}\text{-}(n-2,\kappa(\bar{u})^+)\text{-s-stationary}\}.$$
 \end{claim}
 \begin{proof}[Proof of the Claim] Let $\bar{u}\in \Gamma^*\cap \MS_G$. Note that $\bar{u}$ is good and $lh(\bar{u})$ is $(n-2, \kappa(\bar{u}))$-$s$-stationary. We can then apply the induction hypothesis $\phi^{\bar{u}}_{2,n-1}$ to conclude that $O(\Gamma'\cap \MS_G\cap V_{\kappa(\bar{u})})=H\cap \kappa(\bar{u})$ is $n-1$-$d$-stationary in $V[G\restriction \bar{u}]$ and thus in $V[G]$, as the upper part of the forcing does not add subsets to $\kappa(\bar{u})$. It follows that $\kappa(\bar{u})\in Tr_{n-1}^d(H)$. 
 
 For the other direction, clearly each $\alpha=\kappa(\bar{u})\in Tr_{n-1}^d(H)=Tr_{n-1}^d(O(\Gamma'\cap \MS_G))$
 is a limit point of $O(\MS_G)$ and therefore in $O(\MS_G)$. Hence $\bar{u}\in \MS_G$. We note that if $O(\Gamma'\cap \MS_G)\cap V_{\kappa(\bar{u})}$ is
 $(n-1)$-$d$-stationary, then by $\phi^{\bar{u}}_{0,n-1}$, $lh(\bar{u})$ is $(n-2,\kappa(\bar{u})^+)$-$s$-stationary.
 Apply $\phi_{1,n-1}^{\bar{u}}$ to conclude that $ \Gamma'\cap V_{\kappa(\bar{u})}$ is $\bar{u}$-$(n-2,\kappa(\bar{u})^+)$-$s$-stationary.
 As a result, $\kappa(\bar{u})\in \Gamma^*$.   \end{proof}

 It suffices to show that $\Gamma^*-\Gamma$ is $\bar{U}$-null i.e. $Ind_{\bar{U}}(\Gamma^*-\Gamma)=\emptyset$. Indeed, this will imply that $Tr_{n-1}^d(H)=O(\Gamma^*\cap \MS_G)\subseteq^* O(\Gamma\cap \mathcal{MS}_G)$ so $Tr_{n-1}^d(H)\cap S$ will be bounded. Fix $\xi<lh(\bar{U})$. By the definition, if $\bar{U}\restriction \xi\in j(\Gamma^*)$, then $\Gamma'=j(\Gamma')\cap V_\kappa$ is $\bar{U}\restriction \xi$-$(n-2,\kappa^+)$-$s$-stationary, so  $\xi\in Tr^{\kappa^+}_{n-2}(Ind_{\bar{U}}(\Gamma'))=Tr^{\kappa^+}_{n-2}(T_0)\subseteq Ind_{\bar{U}}(\Gamma)$, so $\bar{U}\restriction \xi\in j(\Gamma)$.
\end{itemize}
\end{proof}

\begin{proof}[Proof of $\phi_{2,n}$]
Let $p=p_0^\frown (\bar{U}, A)$. By Proposition \ref{Prop: reducing to bounded names}, we may assume that for each $\bar{w}\in A$, there is an $R_{\bar{w}}$-name $f(\bar{w})$ such that $p^\frown \bar{w}$ forces $\dot{T}\cap (\kappa(\bar{w})+1)=f(\bar{w})$. Furthermore, we may assume there exists a measure function $b$ satisfying the following: for each $\bar{w}\in A$, and any $r\in R_{<\kappa(\bar{w})}$, there exists a direct extension $r'$ of $r$ in $R_{<\kappa(\bar{w})}$ such that ${r'}^\frown (\bar{w}, b(\bar{w}))$ decides the statement $\kappa(\bar{w})\in f(\bar{w})$.
Split $A$ into two sets:
 $$A_1=\{\bar{u}\in A\mid \exists t\in R_{<\kappa(\bar{u})}/p_0, \ t^{\smallfrown}(\bar{u},b(\bar{u}))^{\smallfrown}(\bar{U},A)\Vdash \kappa(\bar{u})\in \dot{T}\}, \ \ A_2=A\setminus A_1$$
Note that $\lambda=Ind_{\bar{U}}(A_1)\uplus Ind_{\bar{U}}(A_2)$. 
 \begin{claim}
$A_1$ is  $\bar{U}$-$(n-1,\kappa^+)$-s-stationary.
 \end{claim}
\begin{proof}
    Otherwise,  $Ind_{\bar{U}}(A_1)$ is not $(n-1,\kappa^+)$-s-stationary, and thus $A_2$ is a $\bar{U}-(n-1,\kappa^+)$-s-club. Note that by our construction of $p$ and by definition of $A_2$, for every $\bar{u}\in A_2$ and any $t\in R_{<\kappa(\bar{u})}/p_0$, there is a direct extension $t'$ of $t$ such that
    $${t'}^\frown (\bar{u}, b(\bar{u}))^\frown (\bar{U},A\setminus V_{\kappa(\bar{u})+1})\Vdash \kappa(\bar{u})\not\in \dot{T}.$$
   Let $H$ be any generic with $p\in H$. Since $p\Vdash \dot{T}$ is $n$-$d$-stationary, we have $V[H]\models(\dot{T})_H$ is $n$-$d$-stationary and by $\phi_{1,n}$, $(\dot{T})_H\cap O(A_2\cap \mathcal{MS}_H)\neq\emptyset$. Hence we can find $\bar{u}\in A_2$ and a condition $p':=t'^{\smallfrown}( \bar{u},a)^{\smallfrown} \vec{s}^{\smallfrown}(\bar{U},A')\in H/p$ such that $p'\Vdash \kappa(\bar{u})\in \dot{T}$. By the definition of $A_2$, there is a direct extension $t^*$ of $t'$ in $R_{<\kappa(\bar{u})}$ such that $t^{*\smallfrown}(\bar{u},b(\bar{u}))^\smallfrown(\bar{U},A\setminus V_{\kappa(\bar{u}+1})\Vdash \kappa(\bar{u})\notin \dot{T}$. So the condition
 $$p^*=t^{*\smallfrown}(\bar{u},b(\bar{u})\cap a)^{\smallfrown}\vec{s}^\smallfrown(\bar{U},A')$$
 forces both $\kappa(\bar{u})\in \dot{T}$ and $\kappa(\bar{u})\notin \dot{T}$, which is a contradiction.
\end{proof}
 Let $S_1=Ind_{\bar{U}}(A_1)\notin \mathrm{NS}^{(n-1,\kappa^+)}_\lambda$. For each $\xi\in S_1$, $U\restriction\xi\in j(A_1)$. By elementarity of $j$ and the definition of $A_1$, find some $t_\xi\in R_{<\kappa}$ such that $t_\xi\leq_{R_{<\kappa}} p_0 $ and $t_\xi^\frown (\bar{U}\restriction \xi, j(b)(\bar{U}\restriction\xi))^{\frown}(j(\bar{U}),j(A)\setminus V_{\kappa+1})\Vdash_{j(R_{\bar{U}})} \kappa\in j(\dot{T})$.
 
 By the $\kappa^+$-completeness of $\mathrm{NS}^{(n-1,\kappa^+)}_\lambda$, we can find $e_0 \in R_{<\kappa}/p_0$ such that $Z_{e_0}=_{def}\{\bar{u}\in A: e_0^\frown (\bar{u},b(\bar{u}))^\frown (\bar{U},A)\Vdash \kappa(\bar{u})\in \dot{T}\}$ is $\bar{U}$-$(n-1,\kappa^+)$-s-stationary. 
Then $e_0$ is the desired lower part. All that is left to do is to shrink the measure one set.

Let us say that $\overrightarrow{\eta} \in A^{<\omega}$ is \emph{nice}, if $Z_{e_0}\setminus \vec{\eta}$ is $\bar{U}$-$(n-1,\kappa^+)$-s-stationary. We next show that we can find a $\bar{U}$-measure one set $B$ such any $\overrightarrow{\eta}\in B^{<\omega}$ is nice.
We achieve the task in steps by inducting on the length of the finite sequence of measure sequences. 

Let us first check that $A_1=\{\bar{w}: \langle\bar{w} \rangle\text{ is nice}\}\in \bigcap \bar{U}$. Suppose for the sake of contradiction that this set is not in $U(\xi)$ for some $\xi<lh(\bar{U})$. Let $S=Ind_{\bar{U}}(Z_{e_0})$. For each $\gamma\in S-(\xi+1)$, we can find some $\bar{w}_\gamma\in A_1^c\cap j(b)(\bar{U}\restriction \gamma)$ with $j(b)(\bar{U}\restriction \gamma)\cap V_{\kappa(\bar{w}_\gamma)} \in \bigcap \bar{w}_\gamma$. Note that such a $\bar{w}_\gamma$ exists since $A_1^c\cap j(b)(\bar{U}\restriction \gamma)\in U(\xi)$.
 Since $\mathrm{NS}^{(n-1,\kappa^+)}_\lambda$ is $\kappa^+$-complete, there are $\bar{w}\in V_\kappa$ and $(n-1,\kappa^+)$-$s$-stationary $S'\subset S$ such that for any $\gamma\in S'$, $\bar{w}_\gamma=\bar{w}$. But then $\bar{w}$ must be nice since for every $\gamma\in S'$, by elementarity and the definition of $Z_{e_0}\setminus\l \bar{w}\r$, $U\restriction \gamma\in j(Z_{e_0}\setminus\l\bar{w}\r)$. Hence $S'\subseteq Ind_{\bar{U}}(Z_{e_0}\setminus \l\bar{w}\r)$, contradicting with the fact that $\bar{w}\not\in A_1$.

Suppose $\overrightarrow{\eta} \in A^n$ is nice, then $A_{\overrightarrow{\eta},n+1}= \{ \bar{w}: \overrightarrow{\eta}^\frown \bar{w}\text{ is nice}\}$ is in $\bigcap\bar{U}$. The argument is similar to the previous step, by looking at the set $Z_{e_0}\setminus \overrightarrow{\eta}$, and again applying the $\kappa^+$-completeness of $\mathrm{NS}^{(n-1,\kappa^+)}_\lambda$. Let $$A_{n+1}=\Delta_{\text{nice } \overrightarrow{\eta} \in (A_n)^n} A_{\overrightarrow{\eta}, n+1}:= $$$$\{\bar{v}\in \mathcal{MS}\cap V_\kappa\mid\forall \text{nice } \overrightarrow{\eta}\in (A_n)^n\cap V_{\kappa(\bar{v})}, \ \bar{\nu}\in A_{\overrightarrow{\eta},n+1}\}.$$ 
Then $A_{n+1}\in \bigcap \bar{U}$. Finally, it is easy to see that $B=\bigcap_{n\in \omega} A_n$ is as desired. Namely, $(e_0^\frown (\bar{U},B),b)$ is an $n$-$d$-stationary witness for $\dot{T}$.
\end{proof}

\begin{proof}[Proof of  $\phi_{3,n}$]
   By $\phi_{1,n}$, we need to prove that $(\dot{T})_G\cap O(\Gamma\cap MS_G)\neq\emptyset$ for every set $\Gamma$ which is a $\bar{U}$-$(n-1,\kappa^+)$-$s$-club,
   whenever $e\in G$. Suppose toward a contradiction that this is not the case and fix $\Gamma$ as above and $e'\leq e$ such that 
   $e'\Vdash \dot{T}\cap O(\Gamma\cap \dot{\mathcal{MS}_G})=\emptyset$. Let $\vec{\eta}'\in B^{<\omega}$ be such that
$e'=e_0'{}^{\smallfrown}\l \bar{U},A'\r\leq^* e\cat {\vec{\eta}'}$ and $\vec{\eta}$ be the part of $\vec{\eta}'$ above $\max(e_0)$. Since $(e,b)$ is an $n$-$d$-stationary witness, the set 
$Z_{e_0}\setminus \vec{\eta}$ is $\bar{U}$-$(n-1,\kappa^+)$-$s$-stationary subset of $\lambda$, and since $\Gamma$ is a $\bar{U}$-$(n-1,\kappa^+)$-$s$-club in $\lambda$,
there is $\xi\in Ind_{\bar{U}}(Z_{e_0}\setminus \vec{\eta})\cap Ind_{\bar{U}}(\Gamma)$. In particular, we can find $\bar{w}\in \Gamma\cap (Z_{e_0}\setminus \vec{\eta})\cap A'$ such that $A'\cap V_{\kappa(\bar{w})}\in\bigcap\bar{w}$.
Consider the condition 
$$e^*=e'_0{}^{\smallfrown}\l \bar{w},b(\bar{w})\cap A'\r^{\smallfrown}\l\bar{U}, A'\r.$$

Then $e^*\leq e'$ and also $e^*$ is compatible with $e_0^{\smallfrown}\l \bar{w},b(\bar{w})\r^{\smallfrown}\l \bar{U},A\r$. So there exists an extension of $e^*$ that forces the following:
\begin{enumerate}
    \item $\dot{T}\cap O(\Gamma\cap\dot{\MS_G})=\emptyset$.
    \item $\kappa(\bar{w})\in \dot{T}$.
    \item $\bar{w}\in \Gamma\cap \dot{\mathcal{MS}_G}$,
\end{enumerate}
which is a contradiction.
\end{proof}

    

It remains to prove $\phi_{4,n}$. From now on, assume that $\lambda$ is $(n,\kappa^+)$-$s$-stationary.

 \begin{claim}\label{claim: RemovingMF} Suppose that $(e,b)$ is an $n$-$d$-stationary witness for $\dot{T}$. Then
   there is a $\bar{U}$-$(n,\kappa^+)$-$s$-club $\Gamma$ , such that
for each $\bar{v}\in \Gamma$, $e_0^{\smallfrown}(\bar{v},B\cap V_{\kappa(\bar{v})})^{\smallfrown}(\bar{U},B\setminus V_{\kappa(\bar{v})+1})\Vdash \dot{T}\cap \kappa(\bar{v})$ is $n$-d-stationary. In particular, in the generic extension $V[G]$ with $e\in G$, $O(\Gamma\cap \MS_G)\subseteq Tr^d_n(T)$. 
\end{claim}
\begin{proof}

By definition of an $n$-$d$-stationary witness, we have that for each $\vec{\eta}\in B^{<\omega}$,  $Ind_{\bar{U}}(Z_{e_0}\setminus\vec{\eta})\notin NS^{(n-1,\kappa^+)}_\lambda$ and therefore $Tr_{n-1}^{\kappa^+}(Ind(Z_{e_0}\setminus\vec{\eta}))\in Cub^{(n,\kappa^+)}_\lambda${}\footnote{Indeed, $\lambda\setminus Tr_{n-1}^{\kappa^+}(Ind_{\bar{U}}(Z_{e_0}\setminus\vec{\eta}))$ is not $(n,\kappa^+)$-$s$-stationary as witnessed by $ Ind_{\bar{U}}(Z_{e_0}\setminus\vec{\eta})$.}. Since $NS^{(n,\kappa^+)}_\lambda$ is $\kappa^+$-complete and $\lambda$ is $(n,\kappa^+)$-$s$-stationary, we have $$C:=\bigcap_{\vec{\eta}\in B^{<\omega}} Tr_{n-1}^{\kappa^+}(Ind_{\bar{U}}(Z_{e_0}\setminus\vec{\eta}))\in Cub^{(n,\kappa^+)}_\lambda.$$ 
Let $$\Gamma=\{\bar{u}\in B\mid \forall\vec{\eta}\in B^{<\omega}\cap V_{\kappa(\bar{u})}, Z_{e_0}\setminus\vec{\eta}\cap V_{\kappa(\bar{u})}\text{ is }\bar{u}\text{-}(n-1,\kappa(\bar{u})^+)\text{-$s$-stationary}\}.$$
To see that  $\Gamma$ is a $\bar{U}-(n,\kappa^+)$-$s$-club, fix any $\xi\in C$. Note that for every $\vec{\eta}\in B^{<\omega}=j(B^{<\omega})\cap V_\kappa$, we have that $j(Z_{e_0}\setminus\vec{\eta})\cap V_\kappa=Z_{e_0}\setminus \vec{\eta}$ and therefore $\xi\in Tr_{n-1}^{\kappa^+}(Ind_{\bar{U}}(Z_{e_0}\setminus \vec{\eta}))$. By definition this means that $Ind_{\bar{U\restriction \xi}}(Z_{e_0}\setminus\vec{\eta})$ is $(n-1,\kappa^+)$-$s$-stationary in $\xi= lh(\bar{U}\restriction \xi)$, hence $\bar{U}\restriction \xi\in j(\Gamma)$ and $\xi\in Ind_{\bar{U}}(\Gamma)$. This implies that $C\subseteq Ind_{\bar{U}}(\Gamma)$ and thus  $\Gamma$ is a $\bar{U}-(n,\kappa^+)$-$s$-club.

Suppose towards a contradiction that there are $\bar{v}\in \Gamma$ and
$q\leq e_0^{\smallfrown}(\bar{v},B\cap V_{\kappa(\bar{v})})^{\smallfrown}(\bar{U},B)$, $\l \dot{\tau}_i\mid i<\kappa(\bar{v})\r$ a sequence of $R_{\bar{v}}$-names such that 
$q\Vdash \dot{\tau}_i\subseteq \kappa(\bar{v})$ is $(n-1)$-$d$-stationary and $\Delta_{i<\kappa(\bar{v})}Tr_{n-1}^d(\dot{\tau}_i)\cap \dot{T}\cap\kappa(\bar{v})=\emptyset$. Note that by definition of $\Gamma$, we have that $lh(\bar{v})$ is $(n-1,\kappa(\bar{v})^+)$-s-stationary, and $\bar{v}$ is good. By $\phi_{4,n-1}^{\bar{v}}$,
 there is an extension $q'\leq q$ and a $\bar{v}$-$(n-1,\kappa(\bar{v})^+)$-$s$-club set $\Gamma_0$ such that $q'\Vdash O(\Gamma_{0}\cap \dot{\MS_G})\subseteq \Delta_{i<\kappa(\bar{v})} Tr^d_{n-1}(\dot{\tau}_i)$. As $\l\dot{\tau}_i: i<\kappa(\bar{v})\r$ is an $R_{\bar{v}}$-name, we may assume $q'-V_{\kappa(\bar{v})+1}=q-V_{\kappa(\bar{v})+1}$.
 
Recall that by the definition of $\Gamma$, $Z_{e_0}\setminus \vec{\eta}\cap V_{\kappa(\bar{v})}$ is $\bar{v}$-$(n-1,\kappa(\bar{v})^+)$-$s$-stationary where $\vec{\eta}$ is the part of $q'$ in $V_{\kappa(\bar{v})}$ above $\max(e_0)$. Thus $q'$ is of the form $q_0'^\frown \vec{\eta}^\frown (\bar{U}, D)$, where $q_0'\leq_{R_{\max (e_0)}} e_0$ . Hence we can find $\xi\in Ind_{\bar{v}}(\Gamma_0)\cap Ind_{\bar{v}}(Z_{e_0}\setminus\vec{\eta}\cap V_{\kappa(\bar{v})})$. Therefore, there is $\bar{u}\in \Gamma_0\cap Z_{e_0}\setminus\vec{\eta}\cap V_{\kappa(\bar{v})}\cap D$ such that $D\cap V_{\kappa(\bar{u})}\in\bigcap\bar{u}$. Form the condition
$$r=q'_0{}^{\smallfrown}\vec{\eta}^{\smallfrown}(\bar{u},b(\bar{u})\cap D)^{\smallfrown}(\bar{U},D).$$
Notice that $r\leq q'$ and  $r$ is compatible with $e_0^{\smallfrown} (\bar{u},b(\bar{u}))^{\smallfrown}(\bar{U}, B\setminus V_{\kappa(\bar{u})+1})$. It follows that some extension of $r$  forces  $\kappa(\bar{u})\in \dot{T}\cap \Delta_{i<\kappa(\bar{v})}Tr_{n-1}^d(\dot{\tau}_i)\cap \kappa(\bar{v})$, contradicting with the fact that $r\leq q'$ forces that $\dot{T}\cap \Delta_{i<\kappa(\bar{v})}Tr_{n-1}^d(\dot{\tau}_i)\cap \kappa(\bar{v})=\emptyset$.\end{proof}

\begin{proof}[Proof of $\phi_{4,n}$.]

Let $p=p_0^\frown (\bar{U}, A)\Vdash  \langle \dot{S}_i: i<\kappa\rangle$ be an $R_{\bar{U}}$-name for a sequence of $n$-$d$-stationary sets.

For each $i$, let $A_i$ be a maximal antichain subset of $$\{e\in R_{\bar{U}}: \exists \text{ a measure function }b, (e,b) \text{ is an $n$-$d$-stationary witness for $\dot{S}_i$}\}.$$ Such a maximal antichain exists by $\phi_{2,n}$. 
By the $\kappa^+$-c.c of $\mathbb{R}_{\bar{U}}$, each $|A_i|\leq \kappa$. Hence, we can list these conditions as $\{e^i_k: k<\kappa\}$ along with the corresponding witnessing measure functions $\{b^i_k: k<\kappa\}$.

By Claim \ref{claim: RemovingMF}, for each $i,k<\kappa$, there is $B_{i,k}\in Cub^{(n,\kappa^+)}_\lambda$
such that for any $\xi\in B_{i,k}$, $j(p^i_k)\cat {\bar{U}\restriction \xi}\Vdash_{R_{\bar{U}\restriction\xi}} j(\dot{S}_i)\cap \kappa$
is $n$-stationary. Note that this is just a reformulation of the claim in terms of the elementary embedding $j$. By the $\kappa^+$-completeness of $\mathrm{NS}^{n,\kappa^+}_\lambda$ and the fact that $\lambda$ is $(n,\kappa^+)$-$s$-stationary, $B=\bigcap_{i,k<\kappa} B_{i,k}\in Cub^{(n,\kappa^+)}_\lambda$. 

Let $G\subset R_{\bar{U}}$ be generic. In $V[G]$, we define the function $f:\kappa\rightarrow \kappa$. For each $i<\kappa$, let $\rho_i<\kappa$ be the least such that $\rho_i>\max(k,\kappa_0(p^i_k))$ where 
\begin{itemize}
    \item $\kappa_0(p^i_k)=\max{\{\kappa(\bar{v}): \bar{v}\in p^i_k\}\cap \kappa}$,
    \item $p^i_k$ is the unique element in $A_i\cap G$
\end{itemize}
 and define $f(i)=\rho_i$. Let $C_f$ be the club of closure points of $f$. Then by Theorem \ref{Thm: 0-Club}, there is a  $\bar{U}$-$0$-club $\Gamma\subset \mathcal{MS}$  such that $O(\Gamma\cap \mathcal{MS}_G)\subset C_f$.

Finally, consider $$\Gamma^*=\{\bar{u}\in \Gamma: \forall i<\kappa(\bar{u}) \forall k<\kappa(\bar{u}) \ {p^i_k}^\frown \bar{u}\Vdash \dot{S}_i\cap \kappa(\bar{u}) \text{ is $n$-$d$-stationary}\}.$$

Note that $Ind_{\bar{U}}(\Gamma^*)\supset Ind_{\bar{U}}(\Gamma)\cap B$ and hence $\Gamma^*$ is also in $Cub^{(n,\kappa^+)}_\lambda$. To see this, for any $\xi<lh(\bar{U})$ such that $\xi\in B $ and $\bar{U}\restriction \xi \in j(\Gamma)$, We need to check $\bar{U}\restriction \xi \in j(\Gamma^*)$. Fix $i<\kappa$ and $k<\kappa$. Since $\xi\in B$, we know that $j(p^i_k)^\frown \bar{U}\restriction \xi \Vdash_{j(R_{\bar{U}})} j(\dot{S}_i)\cap \kappa$ is $n$-$d$-stationary, as desired.
We claim that $\Gamma^*$ witnesses the lemma, namely that $O(\Gamma^*\cap \mathcal{MS}_G)\subseteq\Delta_{i<\kappa}\text{Tr}^d_n(S_i)$.

For each $\bar{u}\in \Gamma^*\cap \mathcal{MS}_G$ and $i<\kappa(\bar{u})$, we know that $f(i)<\kappa(\bar{u})$. In particular, the unique $p^i_k$ that belongs to $G$ satisfies that $k<\kappa(\bar{u})$. As a result, ${p^i_k}^\frown \bar{u}\in G$ and forces that $\dot{S}_i\cap \kappa(\bar{u})$ is $n$-$d$-stationary. So in $V[G]$, $S_i\cap \kappa(\bar{u})$ is $n$-$d$-stationary for any $i<\kappa(\bar{u})$.\end{proof}

Theorem \ref{theorem: mainstronglimit} now follows easily from the proof in this section and Lemma \ref{Lemma: not weakly compact}.
\begin{remark}
    Here is a comment on the necessity of the goodness assumption on the measure sequence in the proof above. More precisely, without the hypothesis that for any $X\subset \lambda$, there is $f\in V$ such that $j(f)(\kappa)=X$, the statement $\phi_{0,n}$ may not be true. For example, if $\kappa$ is strong in the ground model satisfying GCH and there is no inaccessible cardinal above it. In any Radin extension using a measure sequence whose length is the first repeat point, $\kappa$ will remain measurable (and hence for example $3$-$d$-stationary) but the length of the measure sequence is not $(2, \kappa^+)$-$s$-stationary. 
\end{remark}

\section{Questions}\label{section: questions}

The first question regards the possibility to separate higher order stationary reflection principles from weak compactness in an optimal way:
\begin{question}
    Assuming only the existence of a $n$-stationary cardinal ($n$-$d$-stationary cardinal) $\kappa$ for $n<\omega$, is it consistent that there is a cardinal $\lambda$ which is $n$-stationary ($n$-$d$-stationary cardinal) but not even $\Pi^1_1$-indescribable?
\end{question}

\begin{question}
Is it consistent for a successor cardinal to be $\omega$-stationary?
\end{question}

\begin{problem}
    Characterize the measure sequences $\bar{U}$ such that in $V^{R_{\bar{U}}}$, $\kappa$ is $\Pi^1_{n}$-indescribable, where $n>1$.
\end{problem}

The next question is more open-ended: 
\begin{question}
What other compactness properties can hold at $\kappa$ in the Radin extension $V^{R_{\bar{U}}}$ assuming that the length of the sequence $(\leq 2^\kappa)$ satisfies the certain compactness properties? For example, how about being a J\' onsson cardinal?
\end{question}

\section{Acknowledgment}
We thank the anonymous referee for their comments, suggestions and corrections that greatly improve the paper.

\bibliographystyle{amsplain}
\bibliography{ref}

\end{document}